\documentclass[11pt]{amsart}

\usepackage{amsthm}
\usepackage{amsmath}
\usepackage{amsfonts}
\usepackage{amssymb}
\usepackage{ulem}

\usepackage{graphpap}

\theoremstyle{plain}
\newtheorem{theorem}{\textbf{Theorem}}[section]
\newtheorem{lemma}[theorem]{\textbf{Lemma}}

\newtheorem{claim}{\textbf{Claim}}

\newtheorem*{definition}{\textbf{Definition}}

\newtheorem*{Freiman-3k-4}{\textbf{Freiman $3k-4$ Theorem}}
\newtheorem*{Freiman-3k-3}{\textbf{Freiman $3k-3$ Theorem}}

\newcommand{\Sum}[2]{\underset{#1}{\overset{#2}{\sum}}}

\newcommand{\Z}{\mathbb{Z}}

\newcommand{\be}{\begin{equation}}
\newcommand{\ee}{\end{equation}}
\newcommand{\Summ}[1]{\underset{#1}{\sum}}

\newcommand{\R}{\mathbb{R}}
\newcommand{\ol}[1]{\overline{#1}}
\newcommand{\ber}{\begin{eqnarray}}
\newcommand{\eer}{\end{eqnarray}}
\newcommand{\nn}{\nonumber}
\newcommand{\und}{\;\mbox{ and }\;}

\begin{document}

\title{Inverse Additive Problems for Minkowski Sumsets I}
\author{G.~A.~Freiman} 
\address{The Raymond and Beverly Sackler Faculty of Exact Sciences School of
Mathematical Sciences, Tel Aviv University.}
\email{grisha@post.tau.ac.il}
\author{D.~Grynkiewicz}
\address{Institut f\"{u}r Mathematik und Wissenschaftliches Rechnen. Karl-Franzens-Universit\"{a}t, Graz.}
\email{diambri@hotmail.com}
\thanks{Supported by FWF Grant M1014-N13 }
\author{O. Serra} 
\address{Departament de Matem\`atica Aplicada IV,
        Universitat Polit\`ecnica de Catalunya.}
\email{oserra@ma4.upc.edu}
\thanks{Supported by the  Spanish Research Council    MTM2008-06620-C03-01 and the
Catalan Research Council  2009SGR01387.}%
\author{Y.~V.~Stanchescu}
\address{The Open University of Israel, Raanana 43107 and Afeka Academic College, Tel Aviv 69107.}
\email{ionut@openu.ac.il and yonis@afeka.ac.il}
\thanks{The research of Y.~Stanchescu was supported by the Open University of Israel's Research Fund, Grant No. 100937}

\date{draft, 24 August 2010}

\begin{abstract} We give the structure  of discrete two-dimensional finite sets $A,\,B\subseteq \R^2$
which are extremal for the recently obtained inequality $|A+B|\ge
(\frac{|A|}{m}+\frac{|B|}{n}-1)(m+n-1)$, where $m$ and $n$ are the
minimum number of parallel lines covering $A$ and $B$
respectively.  Via compression techniques, the above bound also
holds when $m$ is the maximal number of points of $A$ contained in
one of the parallel lines covering $A$ and $n$ is the maximal
number of points of $B$ contained in one of the parallel lines
covering $B$. When $m,\,n\geq 2$, we are able to characterize the
case of equality in this bound as well. We also give the structure
of extremal sets in the plane for the projection version of
Bonnesen's sharpening of the Brunn-Minkowski inequality: $\mu
(A+B)\ge (\mu(A)/m+\mu(B)/n)(m+n)$, where $m$ and $n$ are the
lengths of the projections of $A$ and $B$ onto a line.
\end{abstract}

\maketitle

\section{Introduction}

{The Minkowski sum $A+B$, or simply sumset, of two subsets $A$ and
$B$ of an abelian group $G$ is the set of sums $\{ a+b:\; a\in A,\,
b\in B\}$. The classical Brunn--Minkowski inequality gives  a lower
bound for the volume  of the sumset of two convex bodies in $\R^d$
in terms of the volume of the summands. This seminal result in
convex geometry has seen far reaching extensions and applications;
see e.g. \cite{Gardner-survey} for a comprehensive survey. Lower
bounds for the cardinality of the sumset of finite subsets in the
$d$--dimensional Euclidean space have also been given by Freiman
\cite{freiman}, Rusza \cite{ruzsa}, Green and Tao \cite{greentao}, and
Gardner and Gronchi \cite{Gardner-gronchi}.}

{It is known that equality in the Brunn--Minkowski Theorem
holds if and only if the two sets are homothetic. Bonnesen
\cite{bonnesen-original} gave a strengthening of the
Brunn--Minkowski Theorem which takes into account the $(d-1)$-dimensional volume
of the projections of the two sets onto
 a given
hyperplane; see e.g. \cite{bonnesen-book}. For the two-dimensional
case, a discrete analog of the inequality of Bonnesen was obtained
in \cite{gs}, which also improves the previously known discrete lower bounds.
The purpose of this paper is to obtain inverse results which characterize the extremal
sets for both the continuous and
the discrete cases of these inequalities. The characterization is
particularly neat in the two-dimensional case, which is the topic of
this paper. The general $d$-dimensional case is addressed in a
forthcoming paper \cite{fgss_Ddim}.}

Let $A\subseteq \R^2$ be a finite two-dimensional set (i.e., not
contained in a line). It is well-known (see \cite{freiman}) that
$$
|2A|\ge 3|A|-3,
$$
where $|X|$ denotes the cardinality of a finite set $X$.
More generally, for each integer $m \ge 1$, there is an integer $k_0 (m)$
such that, if $|A|\ge k_0 (m)$ and
\begin{equation}\label{eq:dual}
|2A|<\left( 4-\frac{2}{m+1}\right)|A|-(2m+1),
\end{equation}
then $A$ can be covered by at most $m$ parallel lines. For $m=1, 2$, this follows
by results of Freiman \cite{freiman}, and for $m\ge 3$, by results of  Stanchescu
\cite{stanchescu} (see also \cite{stanchescu-2}). The case of addition of two
different sets was considered by Grynkiewicz and Serra \cite{gs}
where they obtained the inequality
\begin{equation}\label{eq:tightdiscrete}
|A+B|\ge \left(\frac{|A|}{m}+\frac{|B|}{n}-1\right)(m+n-1),
\end{equation}
where $A,B \subseteq \R^2$ are finite subsets and $m$ and $n$ are the number of  lines
parallel to some given line $\ell$ which cover $A$ and $B$
respectively.

Our first result, Theorem \ref{thm:invdisc} in Section
\ref{sec:discrete}, shows that equality holds in
\eqref{eq:tightdiscrete} for two-dimensional sets if and only if $A$ and $B$ are both
trapezoids with parallel sides parallel to $\ell$ and
having corresponding pairs of sides between $A$ and $B$ also parallel,
allowing sides consisting of a single point so that our definition of trapezoid also includes triangles.

By basic compression techniques (replicated in the proof of
Theorem \ref{thm-discrete-II}), it can be shown that \eqref{eq:tightdiscrete}
also holds when $m$ is the maximum number of points of $A$
contained on a line parallel to $\ell$ and $n$ is the maximum number of
points of $B$ contained on a line parallel to $\ell$.
Our second result, Theorem \ref{thm-discrete-II} in
Section \ref{sec:disreteII}, characterizes the  more
complicated cases of equality in this alternative bound
for $m,\,n\geq 2$. When $m=1$ or $n=1$, we give an example
showing the structure in these cases can
be much wilder.

Let $A$ and $B$ be two convex two--dimensional bodies in $\R^2$.
The Brunn--Minkowski Theorem implies that
\begin{equation}\label{eq:bmi}
|A+B|\ge \left(|A|^{1/2}+|B|^{1/2}\right)^2,
\end{equation}
where $|X|$ now stands for the area of a measurable set $X$ in
$\R^2$. Moreover, equality holds in (\ref{eq:bmi}) if and only if
$A$ and $B$ are homothetic.
A sharpening of the above Brunn--Minkowski inequality  was obtained
by Bonnesen in 1929 (see e.g. \cite{bonnesen-book}) and reproved in
\cite{gs}:
\begin{equation}\label{eq:bm2i}
|A+B|\ge \left( \frac{|A|}{m}+\frac{|B|}{n}\right) (m+n),
\end{equation}
where $m=|\pi (A)|$ and $n=|\pi (B)|$ are the lengths of the
projections of $A$ and $B$ onto the first coordinate.

Observe that
\begin{equation}\label{eq:imply}\nn
\left( \frac{|A|}{m}+\frac{|B|}{n}\right)
(m+n)=\left(|A|^{\frac{1}{2}}+|B|^{\frac{1}{2}}\right)^2+
\left((\frac{n}{m}|A|)^{\frac{1}{2}}-(\frac{m}{n}|B|)^{\frac{1}{2}}\right)^2\ge
\left(|A|^{\frac{1}{2}}+|B|^{\frac{1}{2}}\right)^2,
\end{equation}
so that the lower bounds (\ref{eq:bmi}) and (\ref{eq:bm2i})
coincide when $|A|/m^2=|B|/n^2$. In this case, we know that
equality holds if and only if $A$ and $B$ are homothetic. On the
other hand, if $|A|/m^2\neq |B|/n^2$, then $A$ and $B$ are not
homothetic, and (\ref{eq:bm2i}) is strictly better than
(\ref{eq:bmi}).  One may ask in this case about the structure of
the sets $A$ and $B$ for which equality holds.  Our third result,
Theorem \ref{thm:invcont-general} in Section \ref{sec:continuous},
shows that equality holds in \eqref{eq:bm2i} if and only if  the
two sets are obtained from a pair of homothetic convex bodies by
stretching them along the vertical line (see the appropriate
definitions in Section \ref{sec:continuous}).

We should remark that the bound \eqref{eq:bm2i}
was originally proven when
$$m=\sup\{|(x+\ell)\cap A|\mid x\in \R^2\}\und n=\sup\{|(x+\ell)\cap B|\mid x\in \R^2\},$$
where $\ell$ is a line parallel to the horizontal axis.
Standard compression or symmetrization techniques are then used to derive
\eqref{eq:bm2i} from this alternative bound,
though it can also be derived independently (via the techniques from \cite{fgss_Ddim,gs}).
The characterization of equality in this fourth bound is addressed in the
forthcoming paper \cite{fgss_Ddim} since the argument is more closely
related to the techniques used for the $d$-dimensional case of equality in
\eqref{eq:tightdiscrete}, which is the main subject of \cite{fgss_Ddim}.
Surprisingly,
the measure-theoretic extremal structures for this original Bonnesen
bound are very well-behaved and exhibit none of the wilder behavior
that is possible for its discrete analog (see Section \ref{sec:disreteII}
and \cite{fgss_Ddim}). It is even more surprising, given that the
discrepancy between the discrete and measure-theoretic extremal structures for
\eqref{eq:tightdiscrete} and \eqref{eq:bm2i} is much milder
(see Sections \ref{sec:discrete} and \ref{sec:continuous}).

\section{The Discrete Case I}\label{sec:discrete}

Let $A\subseteq \R^2$ be a finite set contained in
exactly $m \ge 1$ parallel lines. Then inequality \eqref{eq:dual}
can be rephrased by saying that
\begin{equation}\label{eq:tight}
|2A|\ge \left(4-\frac{2}{m}\right)|A|-(2m-1)=\left(2\frac{|A|}{m}-1\right)(2m-1).
\end{equation}
By choosing the $m$ parallel lines to be $\ell_i:=\{ (x,y)\mid
x=i,\,y\in\R\}$, for $i=0,1,\ldots,m-1$, and letting each
$A_i:=A\cap \ell_i$ be an arithmetic progression with first term
$(i,0)$, difference $d=(0,1)\in \Z^2$ and length $|A|/m\in \Z$,
one can check that inequality (\ref{eq:tight}) becomes tight.

As mentioned in the introduction, the following extension of the
lower bound (\ref{eq:tight}) was given in \cite{gs}. Let $A$ and
$B$ be finite, nonempty sets in $\R^2$. Suppose that, for some
line $\ell$, $A$ and $B$ are covered by exactly $m$ and $n$ lines
parallel to $\ell$ respectively. Then
\begin{equation}\label{eq:tightdiscretebis}
|A+B|\ge \left(\frac{|A|}{m}+\frac{|B|}{n}-1\right)(m+n-1).
\end{equation}
Note that if $A=B$, then $m=n$ and we get inequality
(\ref{eq:tight}).

To state our characterization in Theorem \ref{thm:invdisc}, we need the following
notation
for describing trapezoids and triangles under a common umbrella.
%

\begin{definition} Let $m\geq 1$, $h\geq 1$ be integers and
let $c,\,d\in \R$ be real numbers with $c-d\in \Z$ and
$h-1+(m-1)c\geq (m-1)d$. A {\rm standard trapezoid}
$T(m,h,c,d)$ in $\R^2$ is a translate of the bounded set
$T\subseteq \langle (0,1),(1,d)\rangle\cong \Z^2$
defined by the inequalities $x\geq 0$, $x\leq m-1$, $y\leq cx+h-1$ and $y\geq dx$.
\end{definition}

\begin{figure}[ht]\label{standard-trap-diagram}
\setlength{\unitlength}{4mm}
\begin{center}
\begin{picture}(6,18)

\put(0,0){\line(0,1){18}}
\put(0,18){\line(1,-1){5}}
\put(5,13){\line(0,-1){3}}
\put(0,0){\line(1,2){5}}
\put(-.5,0){\line(1,0){1}}
\put(-.5,18){\line(1,0){1}}
\put(4.5,13){\line(1,0){1}}
\put(4.5,10){\line(1,0){1}}

\put(0,.5){\line(0,-1){1}}
\put(-.2,-2){$0$}
\put(5,.5){\line(0,-1){1}}
\put(3.9,-2){$m-1$}

\put(-3,-.2){0}
\put(-4,17.8){$h-1$}
\put(6.5,12.8){$h-1+(m-1)c$}
\put(6.5,9.8){$(m-1)d$}

\put(0,0){\circle*{0.35}}\put(0,1){\circle*{0.25}}\put(0,2){\circle*{0.25}}\put(0,3){\circle*{0.25}}\put(0,3){\circle*{0.25}}\put(0,4){\circle*{0.25}}
\put(0,5){\circle*{0.25}}\put(0,6){\circle*{0.25}}\put(0,7){\circle*{0.25}}\put(0,7){\circle*{0.25}}\put(0,8){\circle*{0.25}}\put(0,9){\circle*{0.25}}
\put(0,10){\circle*{0.25}}\put(0,11){\circle*{0.25}}\put(0,12){\circle*{0.25}}\put(0,12){\circle*{0.25}}\put(0,13){\circle*{0.25}}\put(0,14){\circle*{0.25}}
\put(0,15){\circle*{0.25}}\put(0,16){\circle*{0.25}}\put(0,17){\circle*{0.25}}\put(0,18){\circle*{0.35}}\put(0,19){\circle*{0.2}}\put(0,-1){\circle*{0.2}}

\put(-1,0){\circle*{0.2}}\put(-1,1){\circle*{0.2}}\put(-1,2){\circle*{0.2}}\put(-1,3){\circle*{0.2}}\put(-1,3){\circle*{0.2}}\put(-1,4){\circle*{0.2}}
\put(-1,5){\circle*{0.2}}\put(-1,6){\circle*{0.2}}\put(-1,7){\circle*{0.2}}\put(-1,7){\circle*{0.2}}\put(-1,8){\circle*{0.2}}\put(-1,9){\circle*{0.2}}
\put(-1,10){\circle*{0.2}}\put(-1,11){\circle*{0.2}}\put(-1,12){\circle*{0.2}}\put(-1,12){\circle*{0.2}}\put(-1,13){\circle*{0.2}}\put(-1,14){\circle*{0.2}}
\put(-1,15){\circle*{0.2}}\put(-1,16){\circle*{0.2}}\put(-1,17){\circle*{0.2}}\put(-1,18){\circle*{0.2}}\put(-1,19){\circle*{0.2}}\put(-1,-1){\circle*{0.2}}

\put(1,0){\circle*{0.2}}\put(1,1){\circle*{0.2}}\put(1,2){\circle*{0.25}}\put(1,3){\circle*{0.25}}\put(1,3){\circle*{0.25}}\put(1,4){\circle*{0.25}}
\put(1,5){\circle*{0.25}}\put(1,6){\circle*{0.25}}\put(1,7){\circle*{0.25}}\put(1,7){\circle*{0.25}}\put(1,8){\circle*{0.25}}\put(1,9){\circle*{0.25}}
\put(1,10){\circle*{0.25}}\put(1,11){\circle*{0.25}}\put(1,12){\circle*{0.25}}\put(1,12){\circle*{0.25}}\put(1,13){\circle*{0.25}}\put(1,14){\circle*{0.25}}
\put(1,15){\circle*{0.25}}\put(1,16){\circle*{0.25}}\put(1,17){\circle*{0.25}}\put(1,18){\circle*{0.2}}\put(1,19){\circle*{0.2}}\put(1,-1){\circle*{0.2}}

\put(2,0){\circle*{0.2}}\put(2,1){\circle*{0.2}}\put(2,2){\circle*{0.2}}\put(2,3){\circle*{0.2}}\put(2,4){\circle*{0.25}}
\put(2,5){\circle*{0.25}}\put(2,6){\circle*{0.25}}\put(2,7){\circle*{0.25}}\put(2,7){\circle*{0.25}}\put(2,8){\circle*{0.25}}\put(2,9){\circle*{0.25}}
\put(2,10){\circle*{0.25}}\put(2,11){\circle*{0.25}}\put(2,12){\circle*{0.25}}\put(2,12){\circle*{0.25}}\put(2,13){\circle*{0.25}}\put(2,14){\circle*{0.25}}
\put(2,15){\circle*{0.25}}\put(2,16){\circle*{0.25}}\put(2,17){\circle*{0.2}}\put(2,18){\circle*{0.2}}\put(2,19){\circle*{0.2}}\put(2,-1){\circle*{0.2}}

\put(3,0){\circle*{0.2}}\put(3,1){\circle*{0.2}}\put(3,2){\circle*{0.2}}\put(3,3){\circle*{0.2}}\put(3,3){\circle*{0.2}}\put(3,4){\circle*{0.2}}
\put(3,5){\circle*{0.2}}\put(3,6){\circle*{0.25}}\put(3,7){\circle*{0.25}}\put(3,7){\circle*{0.25}}\put(3,8){\circle*{0.25}}\put(3,9){\circle*{0.25}}
\put(3,10){\circle*{0.25}}\put(3,11){\circle*{0.25}}\put(3,12){\circle*{0.25}}\put(3,12){\circle*{0.25}}\put(3,13){\circle*{0.25}}\put(3,14){\circle*{0.25}}
\put(3,15){\circle*{0.25}}\put(3,16){\circle*{0.2}}\put(3,17){\circle*{0.2}}\put(3,18){\circle*{0.2}}\put(3,19){\circle*{0.2}}\put(3,-1){\circle*{0.2}}

\put(4,0){\circle*{0.2}}\put(4,1){\circle*{0.2}}\put(4,2){\circle*{0.2}}\put(4,3){\circle*{0.2}}\put(4,3){\circle*{0.2}}\put(4,4){\circle*{0.2}}
\put(4,5){\circle*{0.2}}\put(4,6){\circle*{0.2}}\put(4,7){\circle*{0.2}}\put(4,7){\circle*{0.2}}\put(4,8){\circle*{0.25}}\put(4,9){\circle*{0.25}}
\put(4,10){\circle*{0.25}}\put(4,11){\circle*{0.25}}\put(4,12){\circle*{0.25}}\put(4,12){\circle*{0.25}}\put(4,13){\circle*{0.25}}\put(4,14){\circle*{0.25}}
\put(4,15){\circle*{0.2}}\put(4,16){\circle*{0.2}}\put(4,17){\circle*{0.2}}\put(4,18){\circle*{0.2}}\put(4,19){\circle*{0.2}}\put(4,-1){\circle*{0.2}}

\put(5,0){\circle*{0.2}}\put(5,1){\circle*{0.2}}\put(5,2){\circle*{0.2}}\put(5,3){\circle*{0.2}}\put(5,3){\circle*{0.2}}\put(5,4){\circle*{0.2}}
\put(5,5){\circle*{0.2}}\put(5,6){\circle*{0.2}}\put(5,7){\circle*{0.2}}\put(5,7){\circle*{0.2}}\put(5,8){\circle*{0.2}}\put(5,9){\circle*{0.2}}
\put(5,10){\circle*{0.35}}\put(5,11){\circle*{0.25}}\put(5,12){\circle*{0.25}}\put(5,13){\circle*{0.35}}\put(5,14){\circle*{0.2}}
\put(5,15){\circle*{0.2}}\put(5,16){\circle*{0.2}}\put(5,17){\circle*{0.2}}\put(5,18){\circle*{0.2}}\put(5,19){\circle*{0.2}}\put(5,-1){\circle*{0.2}}

\put(6,0){\circle*{0.2}}\put(6,1){\circle*{0.2}}\put(6,2){\circle*{0.2}}\put(6,3){\circle*{0.2}}\put(6,3){\circle*{0.2}}\put(6,4){\circle*{0.2}}
\put(6,5){\circle*{0.2}}\put(6,6){\circle*{0.2}}\put(6,7){\circle*{0.2}}\put(6,7){\circle*{0.2}}\put(6,8){\circle*{0.2}}\put(6,9){\circle*{0.2}}
\put(6,10){\circle*{0.2}}\put(6,11){\circle*{0.2}}\put(6,12){\circle*{0.2}}\put(6,12){\circle*{0.2}}\put(6,13){\circle*{0.2}}\put(6,14){\circle*{0.2}}
\put(6,15){\circle*{0.2}}\put(6,16){\circle*{0.2}}\put(6,17){\circle*{0.2}}\put(6,18){\circle*{0.2}}\put(6,19){\circle*{0.2}}\put(6,-1){\circle*{0.2}}


\end{picture}
\end{center}
\vspace{3mm} \caption{A Standard Trapezoid $T(m,h,c,d)=T(6,19,-1,2)$.}
\end{figure}
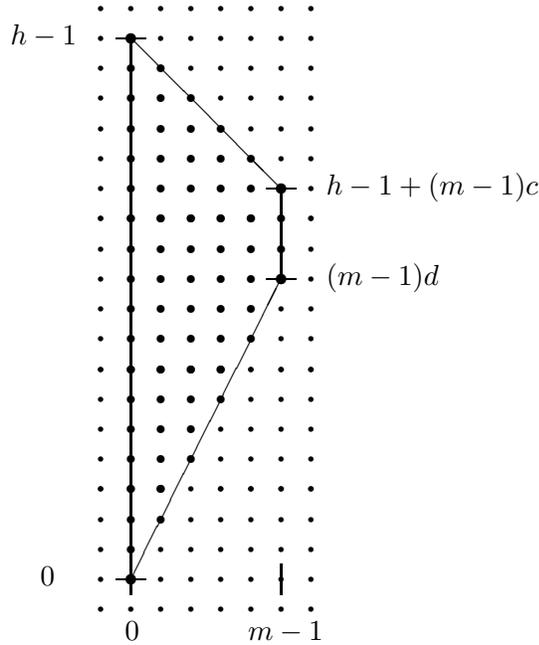

Thus a standard trapezoid $T(m,h,c,d)$ is a trapezoid
with vertical parallel sides, either of which are allowed to consist
of a single point so that our definition of trapezoid includes triangles and
line segments,
and its two (possibly) non-parallel sides of integer slopes $c$ and $d$.
Moreover, $m$ is the number of points in the perpendicular line segment
joining the two parallel sides while $h$ is the number of points in the
leftmost parallel side.  See Figure 1, keeping in mind that the slopes $c$
and $d$ can each be positive, negative or zero in general.

When characterizing equality in \eqref{eq:tightdiscretebis},
by rotating $\R^2$ appropriately and translating $A$ and $B$,
there is no loss of generality to  assume $\ell$ is a vertical
line and $0\in A\cap B$. These are simply normalization hypotheses
in Theorem \ref{thm:invdisc}. When either $A$ or $B$ is one-dimensional,
meaning contained in a line, then the characterization of equality
in \eqref{eq:tightdiscrete} is well-known and follows from Theorem \ref{CDT-forZ}.
A routine calculation shows that two standard trapezoids
$T(m,h,c,d)$ and $T(n,h',c,d)$ with common slopes $c$ and
$d$ satisfy equality \eqref{eq:eq-discrete} below.
Thus the description provided by Theorem \ref{thm:invdisc}
gives a full characterization of the equality \eqref{eq:eq-discrete}.

\begin{theorem}\label{thm:invdisc} Let $A,\,B\subseteq \R^2$ be finite
two-dimensional subsets with $0\in A\cap B$.  Let $m$ and $n$ be the exact number of
vertical lines which cover $A$ and $B$ respectively. Suppose
\begin{equation}\label{eq:eq-discrete}
|A+B|=\left(\frac{|A|}{m}+\frac{|B|}{n}-1\right)(m+n-1).
\end{equation}Then there exists a linear transformation $\varphi:\R^2\rightarrow \R^2$
of the form $$\varphi(x,y)=(\alpha^{-1}x,\beta^{-1}y),$$ for some
positive $\alpha\in \R,\,\beta\in \R$, such that both $\varphi(A)$
and $\varphi(B)$ are standard trapezoids $T(m,h,c,d)$ and
$T(n,h',c,d)$ with common slopes $c$ and $d$.
\end{theorem}

For the proof, we will need the following well-known result
(see \cite[Theorem 1.4, Lemma 1.3, Corollary 8.1]{Natbook}).

\begin{theorem}\label{CDT-forZ}Let $G$ be a torsion-free abelian group and let
 $A,\,B\subseteq G$ be finite, nonempty subsets. Then $$|A+B|\geq |A|+|B|-1$$
 with equality if and only if $\min\{|A|,\,|B|\}=1$ or $A$ and $B$ are
 both arithmetic progressions of common difference.
\end{theorem}

A main part of the proof of Theorem \ref{thm:invdisc} is the following
lemma, which replicates the majority of the original proof used to establish \eqref{eq:tightdiscretebis}.

\begin{lemma}\label{lem:average} Let $I,\,J\subseteq \R$ be finite, nonempty subsets and let $a=(a_i:\; i\in I)$ and $b=(b_j:\; j\in J)$
be two sequences  of non-negative real numbers. For each $t\in I+J$, let
$u_t(a,b)=\max\{a_i+b_{t-i}:\; i\in I, t-i\in J\}$. Then
\be\label{2-red-bound} \frac{1}{|I|+|J|-1}\sum_{t\in I+J}u_t(a,b)\ge
\frac{1}{|I|}\sum_{i\in I} a_i+\frac{1}{|J|}\sum_{j\in J} b_j.\ee
Moreover, if $\min(|I|, |J|)\ge 2$ with $a_i,\,b_j>0$ for $i\in I$ and $j\in J$, then equality holds if and only
if both $I$ and $J$ are arithmetic progressions with   common
difference and both sequences $a$ and $b$ are also arithmetic
progressions with common difference.
\end{lemma}

\begin{proof} For a finite
sequence $x=(x_i:\; i\in K)$, denote by
$\ol{x}=\frac{1}{|K|}\sum_{i\in K}x_i$ its average value. If
$y=(y_i:\; i\in L)$ is another sequence, we denote by   $u^+(x,y)$ the subsequence of the
$|K|+|L|-1$ largest  elements in the sequence $u(x,y)=(u_t(x,y):\;
t\in K+L)$, which is well-defined in view of Theorem \ref{CDT-forZ}.
We shall prove that \be\label{eq:u+} \ol{u^+(a,b)}\ge
\ol{a}+\ol{b}. \ee

Let $m=|I|$ and $n=|J|$. The proof is by induction on $m+n$. If
either $m=1$ or $n=1$, then equality in \eqref{eq:u+} (and in
(\ref{2-red-bound})) clearly holds. Assume that $m,\, n\ge 2$.

Let $\alpha\in I$ and $\beta\in J$ be elements such that $\alpha+\beta\in I+J$ is a unique expression element; for instance,
letting $\alpha$ and $\beta$ be the minimal elements from $I$ and $J$, or the maximal elements, would guarantee this property. Let $a'=(a_i:\; i\in I\setminus \{
\alpha\})$ and $b'=(b_j:\; j\in J\setminus \{ \beta\})$. We may
assume that $\bar{b}-\bar{b'}\leq \bar{a}-\bar{a'}$. We clearly have
    \be\label{eq:ineq1}(m+n-1)\ol{u^+(a,b)}\ge (m+n-2)\ol{u^+(a,b')}+a_{\alpha}+b_{\beta}.\ee
By the induction hypothesis $\ol{u^+(a,b')}\ge \bar{a}+\bar{b'}$
and using the assumption $\bar{b}-\bar{b'}\leq \bar{a}-\bar{a'}$,
it follows that
    \begin{eqnarray}
    (m+n-1)\ol{u^+(a,b)}&\geq&(m+n-2)(\bar{a}+\bar{b'})+a_{\alpha}+b_{\beta}\label{eq:ineq2}\\
    &=&(m+n-2)(\bar{a}+\bar{b'})+m\bar{a}-(m-1)\bar{a'}+n\bar{b}-(n-1)\bar{b'}\nonumber\\
    &=&(m+n-1)(\bar{a}+\bar{b})+(m-1)(\bar{a}-\bar{a'})+(m-1)(\bar{b'}-\bar{b})\nonumber \\
    &\ge&(m+n-1)(\bar{a}+\bar{b}),\label{eq:ineq3}
    \end{eqnarray}
as claimed. In view of $a_i,\,b_j\geq 0$, we see that \eqref{2-red-bound} follows from \eqref{eq:u+}.

Suppose now that $\min \{ m,n\}\ge 2$, that $a_i,\,b_j>0$ for $i\in I$ and $j\in J$, and that there is equality in
\eqref{2-red-bound}. Then  $I$ and $J$ are   arithmetic progressions
with   common difference---as otherwise Theorem \ref{CDT-forZ} implies $|I+J|>|I|+|J|-1$ while $u_t(a,b)>0$ for all $t\in I+J$, whence
inequality \eqref{eq:u+} implies $\ol{a}+\ol{b}\le \ol{u^+(a,b)}<
\frac{1}{|I|+|J|-1}\sum_{t\in I+J}u_t(a,b)$, contrary to assumption. Therefore we may assume w.l.o.g.
that $I=\{1,2,\ldots ,m\}$ and $J=\{ 1,2,\ldots ,n\}$.

In order to see that the sequences $a$ and $b$ are arithmetic
progressions with the same difference,  we again proceed by induction
on $m+n$ starting at $m=n=2$. To this end, suppose $m=2$ and $n\geq 2$. In this case, the equality reads (multiplying both sides by $n+1$)
$$
\frac{n+1}{2}\tilde a+\tilde b+\frac{\tilde b}{n}=a_1+b_1+\Sum{t=2}{n}\max (a_1+b_{t},
a_2+b_{t-1})+a_2+b_n,
$$
where $\tilde b=\Sum{j=1}{n}b_j$ and $\tilde a=\Sum{i=1}{m}a_i=a_1+a_2$. Using the estimate $\max (a_1+b_{t},
a_2+b_{t-1})\geq \frac12(a_1+b_{t}+
a_2+b_{t-1})$, which holds with equality if and only if $a_2-a_1=b_t-b_{t-1}$, we conclude that
\ber\nn\frac{n+1}{2}\tilde a+\tilde b+\frac{\tilde b}{n}&\geq& a_1+b_1+\frac12\Sum{t=2}{n}(a_1+b_{t}+
a_2+b_{t-1})+a_2+b_n\\&=&\label{startermotor}\frac{n+1}{2}\tilde a+\tilde b+\frac{b_1+b_n}{2},\eer with equality if and only if $a$ and $b$ are arithmetic progressions of common difference $a_2-a_1$. When $n=2$, we have $b_1+b_n=\tilde{b}$, so equality holds in \eqref{startermotor}, yielding the desired conclusion. This completes the base case $m=n=2$. Therefore we may assume w.l.o.g. that $m\geq 2$ and $n\geq 3$ and continue with a second base case of sorts.

Suppose $m=2$, $n\geq 3$ and
 $\bar{b}-\bar{b'}\leq \bar{a}-\bar{a'}$ fails both when taking $\alpha=\beta=1$ to be the minimal elements in $I$ and $J$, and when taking $\alpha=m=2$ and $\beta=n$ to be the maximal elements in $I$ and $J$, where $a'$ and $b'$ are as defined in the proof of \eqref{eq:ineq3}. This means
 \begin{align}\nn&\frac{a_2-a_1}{2}=\frac12 (a_1+a_2)-a_1< \frac{1}{n}\Sum{i=1}{n}b_i-\frac{1}{n-1}\Sum{i=2}{n}b_i=\frac{nb_1-\tilde b}{n(n-1)}\;\und \\
 &\frac{a_1-a_2}{2}=\frac12 (a_1+a_2)-a_2<\frac{1}{n}\Sum{i=1}{n}b_i-\frac{1}{n-1}\Sum{i=1}{n-1}b_i=\frac{nb_n-\tilde b}{n(n-1)},\nn\end{align} where, as before, $\tilde b=\Sum{i=1}{n}b_i$.
 Adding both these inequalities yields $0<\frac{nb_1+nb_n-2\tilde b}{n(n-1)}$, which implies $\frac{b_1+b_n}{2}>\frac{\tilde b}{n}$. However, since $m=2$, this contradicts \eqref{startermotor}. So, when  $m=2$ and $n\geq 3$, we may assume $\bar{b}-\bar{b'}\leq \bar{a}-\bar{a'}$ with w.l.o.g. $\alpha=\beta=1$, and when $m\geq 3$ and $n\geq 3$, we may also (by symmetry) assume $\bar{b}-\bar{b'}\leq \bar{a}-\bar{a'}$ with $\alpha=\beta=1$. We can now finish the general case $m\geq 2$ and $n\geq 3$ as follows.

 Since $\bar{b}-\bar{b'}\leq \bar{a}-\bar{a'}$, we must have
equality in \eqref{eq:ineq2} and \eqref{eq:ineq3}. Since $I$ and $J\setminus \{1\}$ are arithmetic progressions with common difference, we have $u^+(a,b')=u(a,b')$. Thus
equality in \eqref{eq:ineq2}  implies
$\ol{u(a,b')}=\ol{u^+(a,b')}=\bar{a}+\bar{b'}$. Consequently, since $n\geq 3$, we can apply the  induction
hypothesis to $a$ and $b'$ and thus conclude they are arithmetic progressions with common
difference $d=a_2-a_1$. Equality in (\ref{eq:ineq3}) implies $\bar{b'}-\bar{b}=\nn\bar{a'}-\bar{a}$, whence $a$ and $b'$ being arithmetic progressions with common difference $d=a_2-a_1$ implies
\begin{align}\nn
&\frac{1}{n}(b_2-b_1)+d\frac{n-2}{2n}=\frac{1}{n-1}\Sum{i=0}{n-2}(b_2+id)-\frac{1}{n}(b_1+\Sum{i=0}{n-2}(b_2+id))
\\&=\bar{b'}-\bar{b}=\nn\bar{a'}-\bar{a}=\frac{1}{m-1}\Sum{i=1}{m-1}(a_1+id)-\frac{1}{m}\Sum{i=0}{m-1}(a_1+id)=
\frac{d}{2}.\end{align} Consequently, $b_2-b_1=d=a_2-a_1$, so that $b$, as well as $b'$ and $a$, is an arithmetic progression with difference $d$, completing the proof.
\end{proof}

\begin{proof}[Proof of Theorem \ref{thm:invdisc}] For a set $X\subseteq \R^2$ and $i\in \R$, we
let $X_i=X\cap \{(x,y)\mid x=i,\,y\in \R\}$ denote the intersection of $X$ with the
vertical line defined by $x=i$. We let $\pi:\R^2\rightarrow \R$ denote the vertical projection map onto the horizontal axis: $\pi(x,y)=x$. Observe that $|\pi(A)|=m$ and $|\pi(B)|=n$. If $m=1$, then $A$ is contained in a vertical line, contrary to the hypothesis that
it is two-dimensional. As $B$ is also two-dimensional by hypothesis, we cannot have $n=1$ either. Therefore  $m,\,n\geq 2$.

We have
\begin{eqnarray}\label{well-align-lowerbound}
 |A+B|&=&\sum_{t\in \pi(A)+\pi(B)} |(A+B)_t|\label{eq:comp} \\
    &=& \sum_{t\in \pi(A)+\pi(B)} \left|\cup_{i\in \pi(A), t-i\in \pi(B)} (A_i+B_{t-i})\right|\label{eq:comp0}\\
    &\ge&  \sum_{t\in \pi(A)+\pi(B)}\max\{ |A_i+B_{t-i}|:\; i\in \pi(A), t-i\in \pi(B)\} \label{eq:comp1} \\
    &\ge&\sum_{t\in \pi(A)+\pi(B)} \max\{ |A_i|+|B_{t-i}|-1:\; i\in \pi(A), t-i\in \pi(B)\}\label{eq:comp2}\\
    &\ge& (m+n-1)\left(\frac{|A|}{m}+\frac{|B|}{n}-1\right)\label{eq:comp3},
        \end{eqnarray}
        where inequality \eqref{eq:comp2} follows from Theorem \ref{CDT-forZ} and inequality \eqref{eq:comp3} follows from the first
        part
        of Lemma \ref{lem:average} with the sequences $a=(|A_i|-\frac12:\;
        i\in \pi(A))$ and $b=(|B_j|-\frac12:\; j\in \pi(B))$.

Since $A$ and $B$ are extremal sets verifying equality in the
lower bound
(\ref{eq:tightdiscrete}), we have equality in each of
(\ref{eq:comp1}), (\ref{eq:comp2}) and (\ref{eq:comp3}).

Equality in (\ref{eq:comp3}) implies, in view of $m,\,n\geq 2$ and the second part of
Lemma \ref{lem:average}, that $\pi(A)$ and $\pi(B)$ are arithmetic
progressions with common difference (say) $\alpha>0$. By applying the linear transformation $(x,y)\mapsto (\alpha^{-1}x,y)$ and then translating, we
may w.l.o.g. assume $\alpha=1$ with $\pi(A)=\{0,1,\ldots ,m-1\}$ and
$\pi(B)=\{0,1,\ldots ,n-1\}$. Moreover, by the same lemma, the
sequences $|A_0|, |A_1|, \ldots,$ $|A_{m-1}|$ and
$|B_0|,|B_1|,\ldots ,|B_{n-1}|$ are also arithmetic progressions
with the same common difference (say) $d'\in \R$, and w.l.o.g. we can assume $d'\geq 0$ (as it suffices to prove the
theorem for horizontal reflections of $A$ and $B$). In particular, for each $t\in \pi(A+B)$,
the terms inside
the $\max$ function in (\ref{eq:comp2}) have the same common value, which means that
\ber\label{wiffleball}|A_i+B_{t-i}|&=&|A_i|+|B_{t-i}|-1\und \\
A_i+B_{t-i}&=&A_j+B_{t-j}\label{union-equals}
\eer whenever $A_i+B_{t-i}$ and $A_j+B_{t-j}$ are nonempty---in view of Theorem \ref{CDT-forZ} and equality holding in \eqref{eq:comp2} and \eqref{eq:comp1}.

For $i=0,1,\ldots,n-1$, let $b_i\in\Z$   be the minimal second coordinate of the elements from $B_i$ and let $b'_i\in \Z$ be the maximal second coordinate of the elements from $B_i$. For $i=0,1,\ldots,m-1$, let $a_i\in \Z$ be the  minimal second coordinate of an element from $A_i$ and let $a'_i\in \Z$ be  the  maximal second coordinate of an element from $A_i$.  In view of \eqref{union-equals}, we have $$a_{i}+b_{t-i}=a_{j}+b_{t-j}\und a'_{i}+b'_{t-i}=a'_{j}+b'_{t-j}$$ whenever $t\in \pi(A+B)=\{0,1,2,\ldots, m+n-2\}$, $i,\,j\in [0,m-1]$ and $t-i,\,t-j\in [0,n-1]$. In particular, since $m,\,n\geq 2$, we have $a_0+b_i=a_1+b_{i-1}$ and $b_0+a_j=b_1+a_{j-1}$ for all $i=1,\ldots,n-1$ and $j=1,\ldots,m-1$. Thus the sequences $b_0,b_1,\ldots,b_{n-1}$ and $a_0,a_1,\ldots,a_{m-1}$ are arithmetic progressions of common difference (say) $d\in \R$ and, likewise, the sequences $b'_0,b'_1,\ldots,b'_{n-1}$ and $a'_0,a'_1,\ldots,a'_{m-1}$ are arithmetic progressions of common difference (say) $c\in \R$.
Since $c-d=b'_1-b'_0=(b_1+|B_1|-b_0-|B_0|)-(b_1-b_0)=|B_1|-|B_0|$, we see that $c-d\in \Z$.

In view of $d'\geq 0$, we have $|B_0|\leq |B_1|\leq \ldots \leq |B_{n-1}|$ and $|A_0|\leq |A_1|\leq \ldots \leq |A_{m-1}|$, with all inequalities being strict when $d'>0$. Thus $|B_1|=1$ is only possible if $d'=0$ and $|B_i|=1$ for all $i=0,1,\ldots,n-1$, in which case $B_i=\{b_i\}$ for all $b_i$. However, in this case, $B$ is contained in a line with slope $d$, contradicting that $B$ is two-dimensional. Therefore, we conclude that $|B_i|\geq 2,$ for $i=1,\ldots,n-1$. Likewise, since $A$ is also assumed two-dimensional by hypothesis, we have $|A_i|\geq 2,$ for $i=1,\ldots,m-1$.

Let $i\in \{1,\ldots,m-1\}$. Then $|A_i|\geq 2$ and $|B_{n-1}|\geq 2$ (as $n\geq 2$), whence the equality $|A_i+B_{n-1}|=|A_i|+|B_{n-1}|-1$ from \eqref{wiffleball} together with Theorem \ref{CDT-forZ} shows that $A_i$ and $B_{n-1}$ are both arithmetic progressions of common difference (say) $\beta\in \R$. If $|A_0|=1$, then $A_0$ is trivially also an arithmetic progression with difference $\beta$, and if $|A_0|\geq 2$, then applying the above argument to $A_0+B_{n-1}$ shows  $A_0$ to be an arithmetic progression with difference $\beta$ as well. Repeating these arguments for $j\in \{0,1,\ldots,n-1\}$, using $A_{m-1}+B_j$ instead of $A_i+B_{n-1}$, shows that each $B_j$, for $j=0,1,\ldots, n-1$, is also an arithmetic progression with difference $\beta$. Since $|B_{n-1}|\geq 2$, we conclude that $\beta\neq 0$. Thus, choosing the sign of the difference appropriately, we may assume $\beta>0$, and then, applying the affine transformation $(x,y)\mapsto (x,\beta^{-1}y)$, we may w.l.o.g assume $\beta=1$. We now have $A,B\subseteq \langle (0,\beta),(\alpha,d)\rangle =\langle (0,1),(1,d)\rangle$, whence the theorem is easily seen to hold in view of $b_0,b_1,\ldots,b_{n-1}$ and $a_0,a_1,\ldots,a_{m-1}$ being arithmetic progressions of common difference $d\in \R$, \ $b'_0,b'_1,\ldots,b'_{n-1}$ and $a'_0,a'_1,\ldots,a'_{m-1}$ being arithmetic progressions of common difference  $c\in \R$, \ $c-d\in \Z$, and $\pi(A)$ and $\pi(B)$ being arithmetic progression of common difference $\alpha=1$.
\end{proof}

\section{The Discrete Case II}\label{sec:disreteII}

As mentioned in the introduction, we also have the
alternative extension of the lower bound \eqref{eq:tightdiscretebis}
in which, for a given line $\ell$,  the meaning of $m$ and $n$ in
that equation  is replaced by the maximum number of points of $A$
contained in a line parallel to $\ell$ and the
maximal number of points of  $B$ contained on a line parallel to
$\ell$; see Theorem \ref{thm-discrete-II} and \eqref{eq:eq-discreteII}.

Theorem \ref{thm-discrete-II} below provides the characterization of
equality for two-dimensional subsets in the alternative bound
\eqref{eq:eq-discreteII} for $m,\,n\geq 2$. As with Theorem
\ref{thm:invdisc}, the question of characterizing equality when one
of the subsets is one-dimensional is well-known and a simple
consequence of Theorem \ref{CDT-forZ}. For the proof, we will
essentially reduce the problem to the characterization of equality
in \eqref{eq:tightdiscretebis} and then invoke Theorem
\ref{thm:invdisc}. As with Theorem \ref{thm:invdisc}, the assumption
about $\ell$ being a horizontal line is purely a normalization
assumption.

The assumption $m,\,n\geq 2$ is quite necessary. When $m=1$ and
$n\geq 2$ (or $n=1$ and $m\geq 2$), there appear to be much wilder
sets satisfying the equality \eqref{eq:eq-discreteII}. For example,
the pair \begin{align}&A=\{(0,0), (0,1), (1,-1)\}\nn\und\\\nn
&B=\{(0,2),(0,1),(0,0),
(1,0),(1,-1),(2,0),(2,-1),(2,-2),(x,0)\},\end{align} where $x\geq
4$, attains equality in \eqref{eq:eq-discreteII} for
$m=1, n=4, |A|=3, |B|=9, |A+B|=17,$ yet the
horizontal sections of $B$ are not all arithmetic progressions nor
is the set $B$ even vaguely convex in appearance!

Throughout this section, we use $\pi:\R^2\rightarrow \R$ and
$\pi':\R^2\rightarrow \R$ to denote the vertical and horizontal
projection maps: $$\pi(x,y)=x\quad \und\quad  \pi'(x,y)=y.$$

As alluded to in the introduction, the extremal structures for
\eqref{eq:eq-discreteII} include two new cases which are particular
perturbations of a trapezoid.

\begin{definition}
Let $A'$ be a standard trapezoid $T(m,h,c,d)$ with $\min
\pi'(A')=\min \pi(A')=0$ and $c,\,d\geq 0$ integers with $c$ and $d$ not both
zero. Given a $(0,1)$--sequence $\epsilon=(\epsilon_i:\; i\in
\Z)$ with $\epsilon_i=0$ for all sufficiently small $i$, let $\tau_{\epsilon}:\Z^2\rightarrow \Z^2$ be the
map defined by
$$\tau_{\epsilon}(x,y)=(x+\sum_{i\le y}\epsilon_i,y).$$

Let $\epsilon=(\epsilon_i:\; i\in
\Z)$ be a $(0,1)$--sequence with zero entries outside the
interval $[md,h-c-1]$ and  at most one entry equal to one in every
subsequence of consecutive $\max\{ c,d\}$ entries. An
$\epsilon$--standard trapezoid  $T_{\epsilon}(m,h,c,d)$ is a
translate of $\tau_{\epsilon}(A')$.
\end{definition}

An example of an $\epsilon$--standard trapezoid  is shown in Figure
2. As  can be checked by adapting the arguments from the proof of
Theorem \ref{thm-discrete-II}, if a pair $(A',B)$ of standard
trapezoids satisfy equality in \eqref{eq:tightdiscrete}, then
the pair $(A,B)$, with $A$ an $\epsilon$--standard trapezoid with
the same parameters as $A'$, satisfies equality in
\eqref{eq:eq-discreteII} for any suitable choice of the sequence
$\epsilon$.  However, the projection of $A$ along any line may have
more than $m$ points and $A$ may not even be convex; see for
instance the example in  Figure 2.

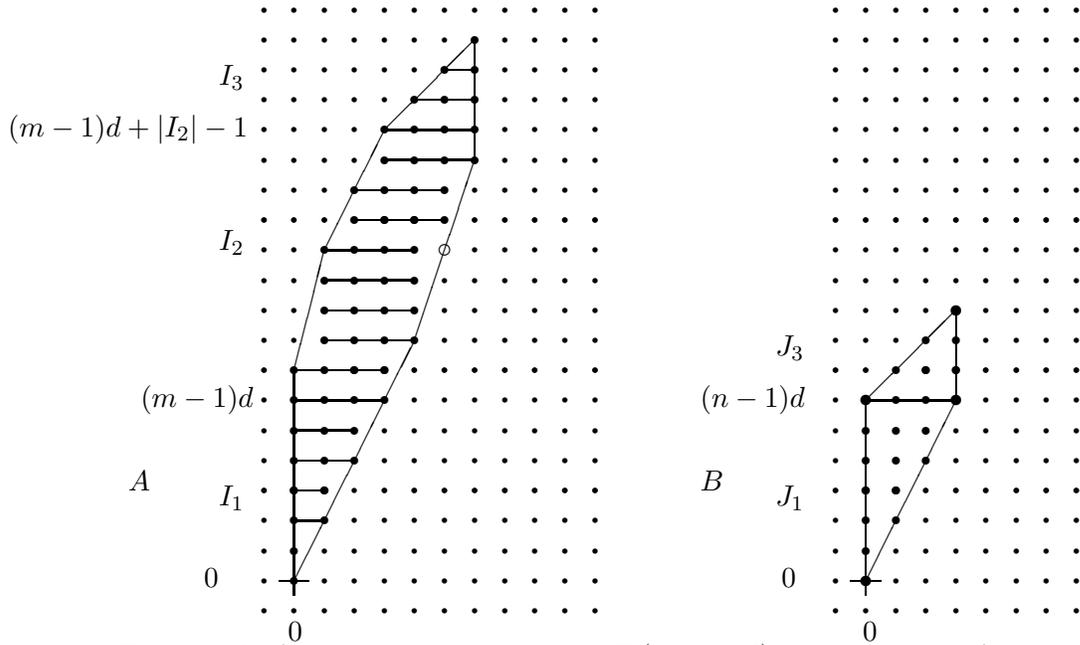
\begin{figure}[ht]\label{fig:apptrap}
\setlength{\unitlength}{4mm}
\begin{center}
\begin{picture}(20,19)

\put(0,0){\line(0,1){7}}
\put(0,0){\line(1,2){3}} \put(0,7){\line(1,4){1}}
\put(1,11){\line(1,2){2}} \put(3,6){\line(1,2){1}}
\put(4,8){\line(1,3){2}}

\put(0,6){\line(1,0){3}} \put(0,7){\line(1,0){3}}
\put(1,8){\line(1,0){3}} \put(1,9){\line(1,0){3}}
\put(1,10){\line(1,0){3}} \put(1,11){\line(1,0){3}}
\put(2,12){\line(1,0){3}} \put(2,13){\line(1,0){3}}
\put(3,14){\line(1,0){3}} \put(3,15){\line(1,0){3}}
\put(3,15){\line(1,1){3}} \put(6,14){\line(0,1){4}}
\put(4,16){\line(1,0){2}} \put(5,17){\line(1,0){1}}
\put(0,5){\line(1,0){2}} \put(0,4){\line(1,0){2}}
\put(0,3){\line(1,0){1}} \put(0,2){\line(1,0){1}}

\put(-5.5,3){$A$} \put(13.5,3){$B$}

\put(-.5,0){\line(1,0){1}} \put(18.5,0){\line(1,0){1}}

\put(19,0){\line(0,1){6}} \put(19,6){\line(1,1){3}}
\put(22,6){\line(0,1){3}} \put(19,0){\line(1,2){3}}

\put(0,.5){\line(0,-1){1}} \put(-.2,-2){$0$} \put(18.9,-2){$0$}

\put(-2.5,11){$I_2$} \put(-2.5,2.5){$I_1$} \put(-2.5,16.5){$I_3$}

\put(16,7.5){$J_3$} \put(16,2.5){$J_1$}

\put(19,6){\line(1,0){3}}

\put(-5.1,5.8){$(m-1)d$} \put(13.5,5.8){$(n-1)d$}

\put(-9.5,14.8){$(m-1)d+|I_2|-1$}

\put(19,.5){\line(0,-1){1}}

\put(16.2,-.2){$0$} \put(-3,-.2){$0$}

\put(0,0){\circle*{0.25}}\put(0,1){\circle*{0.25}}\put(0,2){\circle*{0.25}}\put(0,3){\circle*{0.25}}\put(0,3){\circle*{0.25}}\put(0,4){\circle*{0.25}}
\put(0,5){\circle*{0.25}}\put(0,6){\circle*{0.25}}\put(0,7){\circle*{0.25}}\put(0,7){\circle*{0.2}}\put(0,8){\circle*{0.2}}\put(0,9){\circle*{0.2}}
\put(0,10){\circle*{0.2}}\put(0,11){\circle*{0.2}}\put(0,12){\circle*{0.2}}\put(0,12){\circle*{0.2}}\put(0,13){\circle*{0.2}}\put(0,14){\circle*{0.2}}
\put(0,15){\circle*{0.2}}\put(0,16){\circle*{0.2}}\put(0,17){\circle*{0.2}}\put(0,18){\circle*{0.2}}\put(0,19){\circle*{0.2}}\put(0,-1){\circle*{0.2}}

\put(-1,0){\circle*{0.2}}\put(-1,1){\circle*{0.2}}\put(-1,2){\circle*{0.2}}\put(-1,3){\circle*{0.2}}\put(-1,3){\circle*{0.2}}\put(-1,4){\circle*{0.2}}
\put(-1,5){\circle*{0.2}}\put(-1,6){\circle*{0.2}}\put(-1,7){\circle*{0.2}}\put(-1,7){\circle*{0.2}}\put(-1,8){\circle*{0.2}}\put(-1,9){\circle*{0.2}}
\put(-1,10){\circle*{0.2}}\put(-1,11){\circle*{0.2}}\put(-1,12){\circle*{0.2}}\put(-1,12){\circle*{0.2}}\put(-1,13){\circle*{0.2}}\put(-1,14){\circle*{0.2}}
\put(-1,15){\circle*{0.2}}\put(-1,16){\circle*{0.2}}\put(-1,17){\circle*{0.2}}\put(-1,18){\circle*{0.2}}\put(-1,19){\circle*{0.2}}\put(-1,-1){\circle*{0.2}}

\put(1,0){\circle*{0.2}}\put(1,1){\circle*{0.2}}\put(1,2){\circle*{0.25}}\put(1,3){\circle*{0.25}}\put(1,3){\circle*{0.25}}\put(1,4){\circle*{0.25}}
\put(1,5){\circle*{0.25}}\put(1,6){\circle*{0.25}}\put(1,7){\circle*{0.25}}\put(1,8){\circle*{0.25}}\put(1,9){\circle*{0.25}}
\put(1,10){\circle*{0.25}}\put(1,11){\circle*{0.25}}\put(1,12){\circle*{0.2}}\put(1,12){\circle*{0.2}}\put(1,13){\circle*{0.2}}\put(1,14){\circle*{0.2}}
\put(1,15){\circle*{0.2}}\put(1,16){\circle*{0.2}}\put(1,17){\circle*{0.2}}\put(1,18){\circle*{0.2}}\put(1,19){\circle*{0.2}}\put(1,-1){\circle*{0.2}}

\put(2,0){\circle*{0.2}}\put(2,1){\circle*{0.2}}\put(2,2){\circle*{0.2}}\put(2,3){\circle*{0.2}}\put(2,4){\circle*{0.25}}
\put(2,5){\circle*{0.25}}\put(2,6){\circle*{0.25}}\put(2,7){\circle*{0.25}}\put(2,7){\circle*{0.25}}\put(2,8){\circle*{0.25}}\put(2,9){\circle*{0.25}}
\put(2,10){\circle*{0.25}}\put(2,11){\circle*{0.25}}\put(2,12){\circle*{0.25}}\put(2,12){\circle*{0.2}}\put(2,13){\circle*{0.25}}\put(2,14){\circle*{0.2}}
\put(2,15){\circle*{0.2}}\put(2,16){\circle*{0.2}}\put(2,17){\circle*{0.2}}\put(2,18){\circle*{0.2}}\put(2,19){\circle*{0.2}}\put(2,-1){\circle*{0.2}}

\put(3,0){\circle*{0.2}}\put(3,1){\circle*{0.2}}\put(3,2){\circle*{0.2}}\put(3,3){\circle*{0.2}}\put(3,3){\circle*{0.2}}\put(3,4){\circle*{0.2}}
\put(3,5){\circle*{0.2}}\put(3,6){\circle*{0.25}}\put(3,7){\circle*{0.25}}\put(3,7){\circle*{0.25}}\put(3,8){\circle*{0.25}}\put(3,9){\circle*{0.25}}
\put(3,10){\circle*{0.25}}\put(3,11){\circle*{0.25}}\put(3,12){\circle*{0.25}}\put(3,12){\circle*{0.25}}\put(3,13){\circle*{0.25}}\put(3,14){\circle*{0.25}}
\put(3,15){\circle*{0.25}}\put(3,16){\circle*{0.2}}\put(3,17){\circle*{0.2}}\put(3,18){\circle*{0.2}}\put(3,19){\circle*{0.2}}\put(3,-1){\circle*{0.2}}

\put(4,0){\circle*{0.2}}\put(4,1){\circle*{0.2}}\put(4,2){\circle*{0.2}}\put(4,3){\circle*{0.2}}\put(4,3){\circle*{0.2}}\put(4,4){\circle*{0.2}}
\put(4,5){\circle*{0.2}}\put(4,6){\circle*{0.2}}\put(4,7){\circle*{0.2}}\put(4,7){\circle*{0.2}}\put(4,8){\circle*{0.25}}\put(4,9){\circle*{0.25}}
\put(4,10){\circle*{0.25}}\put(4,11){\circle*{0.25}}\put(4,12){\circle*{0.25}}\put(4,12){\circle*{0.25}}\put(4,13){\circle*{0.25}}\put(4,14){\circle*{0.25}}
\put(4,15){\circle*{0.25}}\put(4,16){\circle*{0.25}}\put(4,17){\circle*{0.2}}\put(4,18){\circle*{0.2}}\put(4,19){\circle*{0.2}}\put(4,-1){\circle*{0.2}}

\put(5,0){\circle*{0.2}}\put(5,1){\circle*{0.2}}\put(5,2){\circle*{0.2}}\put(5,3){\circle*{0.2}}\put(5,3){\circle*{0.2}}\put(5,4){\circle*{0.2}}
\put(5,5){\circle*{0.2}}\put(5,6){\circle*{0.2}}\put(5,7){\circle*{0.2}}\put(5,7){\circle*{0.2}}\put(5,8){\circle*{0.2}}\put(5,9){\circle*{0.2}}
\put(5,10){\circle*{0.2}}\put(5,11){\circle{0.35}}\put(5,12){\circle*{0.25}}\put(5,13){\circle*{0.25}}\put(5,14){\circle*{0.25}}
\put(5,15){\circle*{0.25}}\put(5,16){\circle*{0.25}}\put(5,17){\circle*{0.25}}\put(5,18){\circle*{0.2}}\put(5,19){\circle*{0.2}}\put(5,-1){\circle*{0.2}}

\put(6,0){\circle*{0.2}}\put(6,1){\circle*{0.2}}\put(6,2){\circle*{0.2}}\put(6,3){\circle*{0.2}}\put(6,3){\circle*{0.2}}\put(6,4){\circle*{0.2}}
\put(6,5){\circle*{0.2}}\put(6,6){\circle*{0.2}}\put(6,7){\circle*{0.2}}\put(6,7){\circle*{0.2}}\put(6,8){\circle*{0.2}}\put(6,9){\circle*{0.2}}
\put(6,10){\circle*{0.2}}\put(6,11){\circle*{0.2}}\put(6,12){\circle*{0.2}}\put(6,12){\circle*{0.2}}\put(6,13){\circle*{0.2}}\put(6,14){\circle*{0.25}}
\put(6,15){\circle*{0.25}}\put(6,16){\circle*{0.25}}\put(6,17){\circle*{0.25}}\put(6,18){\circle*{0.25}}\put(6,19){\circle*{0.2}}\put(6,-1){\circle*{0.2}}

\put(7,0){\circle*{0.2}}\put(7,1){\circle*{0.2}}\put(7,2){\circle*{0.2}}\put(7,3){\circle*{0.2}}\put(7,3){\circle*{0.2}}\put(7,4){\circle*{0.2}}
\put(7,5){\circle*{0.2}}\put(7,6){\circle*{0.2}}\put(7,7){\circle*{0.2}}\put(7,7){\circle*{0.2}}\put(7,8){\circle*{0.2}}\put(7,9){\circle*{0.2}}
\put(7,10){\circle*{0.2}}\put(7,11){\circle*{0.2}}\put(7,12){\circle*{0.2}}\put(7,12){\circle*{0.2}}\put(7,13){\circle*{0.2}}\put(7,14){\circle*{0.2}}
\put(7,15){\circle*{0.2}}\put(7,16){\circle*{0.2}}\put(7,17){\circle*{0.2}}\put(7,18){\circle*{0.2}}\put(7,19){\circle*{0.2}}\put(7,-1){\circle*{0.2}}

\put(8,0){\circle*{0.2}}\put(8,1){\circle*{0.2}}\put(8,2){\circle*{0.2}}\put(8,3){\circle*{0.2}}\put(8,3){\circle*{0.2}}\put(8,4){\circle*{0.2}}
\put(8,5){\circle*{0.2}}\put(8,6){\circle*{0.2}}\put(8,7){\circle*{0.2}}\put(8,7){\circle*{0.2}}\put(8,8){\circle*{0.2}}\put(8,9){\circle*{0.2}}
\put(8,10){\circle*{0.2}}\put(8,11){\circle*{0.2}}\put(8,12){\circle*{0.2}}\put(8,12){\circle*{0.2}}\put(8,13){\circle*{0.2}}\put(8,14){\circle*{0.2}}
\put(8,15){\circle*{0.2}}\put(8,16){\circle*{0.2}}\put(8,17){\circle*{0.2}}\put(8,18){\circle*{0.2}}\put(8,19){\circle*{0.2}}\put(8,-1){\circle*{0.2}}

\put(9,0){\circle*{0.2}}\put(9,1){\circle*{0.2}}\put(9,2){\circle*{0.2}}\put(9,3){\circle*{0.2}}\put(9,3){\circle*{0.2}}\put(9,4){\circle*{0.2}}
\put(9,5){\circle*{0.2}}\put(9,6){\circle*{0.2}}\put(9,7){\circle*{0.2}}\put(9,7){\circle*{0.2}}\put(9,8){\circle*{0.2}}\put(9,9){\circle*{0.2}}
\put(9,10){\circle*{0.2}}\put(9,11){\circle*{0.2}}\put(9,12){\circle*{0.2}}\put(9,12){\circle*{0.2}}\put(9,13){\circle*{0.2}}\put(9,14){\circle*{0.2}}
\put(9,15){\circle*{0.2}}\put(9,16){\circle*{0.2}}\put(9,17){\circle*{0.2}}\put(9,18){\circle*{0.2}}\put(9,19){\circle*{0.2}}\put(9,-1){\circle*{0.2}}

\put(10,0){\circle*{0.2}}\put(10,1){\circle*{0.2}}\put(10,2){\circle*{0.2}}\put(10,3){\circle*{0.2}}\put(10,3){\circle*{0.2}}\put(10,4){\circle*{0.2}}
\put(10,5){\circle*{0.2}}\put(10,6){\circle*{0.2}}\put(10,7){\circle*{0.2}}\put(10,7){\circle*{0.2}}\put(10,8){\circle*{0.2}}\put(10,9){\circle*{0.2}}
\put(10,10){\circle*{0.2}}\put(10,11){\circle*{0.2}}\put(10,12){\circle*{0.2}}\put(10,12){\circle*{0.2}}\put(10,13){\circle*{0.2}}\put(10,14){\circle*{0.2}}
\put(10,15){\circle*{0.2}}\put(10,16){\circle*{0.2}}\put(10,17){\circle*{0.2}}\put(10,18){\circle*{0.2}}\put(10,19){\circle*{0.2}}\put(10,-1){\circle*{0.2}}

\put(18,0){\circle*{0.2}}\put(18,1){\circle*{0.2}}\put(18,2){\circle*{0.2}}\put(18,3){\circle*{0.2}}\put(18,3){\circle*{0.2}}\put(18,4){\circle*{0.2}}
\put(18,5){\circle*{0.2}}\put(18,6){\circle*{0.2}}\put(18,7){\circle*{0.2}}\put(18,7){\circle*{0.2}}\put(18,8){\circle*{0.2}}\put(18,9){\circle*{0.2}}
\put(18,10){\circle*{0.2}}\put(18,11){\circle*{0.2}}\put(18,12){\circle*{0.2}}\put(18,12){\circle*{0.2}}\put(18,13){\circle*{0.2}}\put(18,14){\circle*{0.2}}
\put(18,15){\circle*{0.2}}\put(18,16){\circle*{0.2}}\put(18,17){\circle*{0.2}}\put(18,18){\circle*{0.2}}\put(18,19){\circle*{0.2}}\put(18,-1){\circle*{0.2}}

\put(19,0){\circle*{0.35}}\put(19,1){\circle*{0.25}}\put(19,2){\circle*{0.25}}\put(19,3){\circle*{0.25}}\put(19,3){\circle*{0.25}}\put(19,4){\circle*{0.25}}
\put(19,5){\circle*{0.25}}\put(19,6){\circle*{0.35}}\put(19,7){\circle*{0.2}}\put(19,7){\circle*{0.2}}\put(19,8){\circle*{0.2}}\put(19,9){\circle*{0.2}}
\put(19,10){\circle*{0.2}}\put(19,11){\circle*{0.2}}\put(19,12){\circle*{0.2}}\put(19,12){\circle*{0.2}}\put(19,13){\circle*{0.2}}\put(19,14){\circle*{0.2}}
\put(19,15){\circle*{0.2}}\put(19,16){\circle*{0.2}}\put(19,17){\circle*{0.2}}\put(19,18){\circle*{0.2}}\put(19,19){\circle*{0.2}}\put(19,-1){\circle*{0.2}}

\put(20,0){\circle*{0.2}}\put(20,1){\circle*{0.2}}\put(20,2){\circle*{0.25}}\put(20,3){\circle*{0.25}}\put(20,3){\circle*{0.25}}\put(20,4){\circle*{0.25}}
\put(20,5){\circle*{0.25}}\put(20,6){\circle*{0.25}}\put(20,7){\circle*{0.25}}\put(20,7){\circle*{0.2}}\put(20,8){\circle*{0.2}}\put(20,9){\circle*{0.2}}
\put(20,10){\circle*{0.2}}\put(20,11){\circle*{0.2}}\put(20,12){\circle*{0.2}}\put(20,12){\circle*{0.2}}\put(20,13){\circle*{0.2}}\put(20,14){\circle*{0.2}}
\put(20,15){\circle*{0.2}}\put(20,16){\circle*{0.2}}\put(20,17){\circle*{0.2}}\put(20,18){\circle*{0.2}}\put(20,19){\circle*{0.2}}\put(20,-1){\circle*{0.2}}

\put(21,0){\circle*{0.2}}\put(21,1){\circle*{0.2}}\put(21,2){\circle*{0.2}}\put(21,3){\circle*{0.2}}\put(21,3){\circle*{0.2}}\put(21,4){\circle*{0.25}}
\put(21,5){\circle*{0.25}}\put(21,6){\circle*{0.25}}\put(21,7){\circle*{0.25}}\put(21,7){\circle*{0.25}}\put(21,8){\circle*{0.25}}\put(21,9){\circle*{0.2}}
\put(21,10){\circle*{0.2}}\put(21,11){\circle*{0.2}}\put(21,12){\circle*{0.2}}\put(21,12){\circle*{0.2}}\put(21,13){\circle*{0.2}}\put(21,14){\circle*{0.2}}
\put(21,15){\circle*{0.2}}\put(21,16){\circle*{0.2}}\put(21,17){\circle*{0.2}}\put(21,18){\circle*{0.2}}\put(21,19){\circle*{0.2}}\put(21,-1){\circle*{0.2}}

\put(22,0){\circle*{0.2}}\put(22,1){\circle*{0.2}}\put(22,2){\circle*{0.2}}\put(22,3){\circle*{0.2}}\put(22,3){\circle*{0.2}}\put(22,4){\circle*{0.2}}
\put(22,5){\circle*{0.2}}\put(22,6){\circle*{0.35}}\put(22,7){\circle*{0.25}}\put(22,7){\circle*{0.25}}\put(22,8){\circle*{0.25}}\put(22,9){\circle*{0.35}}
\put(22,10){\circle*{0.2}}\put(22,11){\circle*{0.2}}\put(22,12){\circle*{0.2}}\put(22,12){\circle*{0.2}}\put(22,13){\circle*{0.2}}\put(22,14){\circle*{0.2}}
\put(22,15){\circle*{0.2}}\put(22,16){\circle*{0.2}}\put(22,17){\circle*{0.2}}\put(22,18){\circle*{0.2}}\put(22,19){\circle*{0.2}}\put(22,-1){\circle*{0.2}}

\put(23,0){\circle*{0.2}}\put(23,1){\circle*{0.2}}\put(23,2){\circle*{0.2}}\put(23,3){\circle*{0.2}}\put(23,3){\circle*{0.2}}\put(23,4){\circle*{0.2}}
\put(23,5){\circle*{0.2}}\put(23,6){\circle*{0.2}}\put(23,7){\circle*{0.2}}\put(23,7){\circle*{0.2}}\put(23,8){\circle*{0.2}}\put(23,9){\circle*{0.2}}
\put(23,10){\circle*{0.2}}\put(23,11){\circle*{0.2}}\put(23,12){\circle*{0.2}}\put(23,12){\circle*{0.2}}\put(23,13){\circle*{0.2}}\put(23,14){\circle*{0.2}}
\put(23,15){\circle*{0.2}}\put(23,16){\circle*{0.2}}\put(23,17){\circle*{0.2}}\put(23,18){\circle*{0.2}}\put(23,19){\circle*{0.2}}\put(23,-1){\circle*{0.2}}

\put(24,0){\circle*{0.2}}\put(24,1){\circle*{0.2}}\put(24,2){\circle*{0.2}}\put(24,3){\circle*{0.2}}\put(24,3){\circle*{0.2}}\put(24,4){\circle*{0.2}}
\put(24,5){\circle*{0.2}}\put(24,6){\circle*{0.2}}\put(24,7){\circle*{0.2}}\put(24,7){\circle*{0.2}}\put(24,8){\circle*{0.2}}\put(24,9){\circle*{0.2}}
\put(24,10){\circle*{0.2}}\put(24,11){\circle*{0.2}}\put(24,12){\circle*{0.2}}\put(24,12){\circle*{0.2}}\put(24,13){\circle*{0.2}}\put(24,14){\circle*{0.2}}
\put(24,15){\circle*{0.2}}\put(24,16){\circle*{0.2}}\put(24,17){\circle*{0.2}}\put(24,18){\circle*{0.2}}\put(24,19){\circle*{0.2}}\put(24,-1){\circle*{0.2}}

\put(25,0){\circle*{0.2}}\put(25,1){\circle*{0.2}}\put(25,2){\circle*{0.2}}\put(25,3){\circle*{0.2}}\put(25,3){\circle*{0.2}}\put(25,4){\circle*{0.2}}
\put(25,5){\circle*{0.2}}\put(25,6){\circle*{0.2}}\put(25,7){\circle*{0.2}}\put(25,7){\circle*{0.2}}\put(25,8){\circle*{0.2}}\put(25,9){\circle*{0.2}}
\put(25,10){\circle*{0.2}}\put(25,11){\circle*{0.2}}\put(25,12){\circle*{0.2}}\put(25,12){\circle*{0.2}}\put(25,13){\circle*{0.2}}\put(25,14){\circle*{0.2}}
\put(25,15){\circle*{0.2}}\put(25,16){\circle*{0.2}}\put(25,17){\circle*{0.2}}\put(25,18){\circle*{0.2}}\put(25,19){\circle*{0.2}}\put(25,-1){\circle*{0.2}}

\put(26,0){\circle*{0.2}}\put(26,1){\circle*{0.2}}\put(26,2){\circle*{0.2}}\put(26,3){\circle*{0.2}}\put(26,3){\circle*{0.2}}\put(26,4){\circle*{0.2}}
\put(26,5){\circle*{0.2}}\put(26,6){\circle*{0.2}}\put(26,7){\circle*{0.2}}\put(26,7){\circle*{0.2}}\put(26,8){\circle*{0.2}}\put(26,9){\circle*{0.2}}
\put(26,10){\circle*{0.2}}\put(26,11){\circle*{0.2}}\put(26,12){\circle*{0.2}}\put(26,12){\circle*{0.2}}\put(26,13){\circle*{0.2}}\put(26,14){\circle*{0.2}}
\put(26,15){\circle*{0.2}}\put(26,16){\circle*{0.2}}\put(26,17){\circle*{0.2}}\put(26,18){\circle*{0.2}}\put(26,19){\circle*{0.2}}\put(26,-1){\circle*{0.2}}

\end{picture}
\end{center}
\vspace{3mm} \caption{An $\epsilon$--standard trapezoid
$T_{\epsilon}(4,16,1,2)$, with $\epsilon_i=1$ for $i\in
\{8,12,14\}$, is shown on the left side of the figure. Together with
the standard trapezoid  $T(4,7,1,2)$ displayed in the right side,
they form an instance of Theorem \ref{thm-discrete-II}(b).}
\end{figure}

There is a third particular case in Theorem \ref{thm-discrete-II}
besides the $\epsilon$--standard trapezoids,  an example of which
is displayed in Figure \ref{case (c)}.
It is described by Theorem \ref{thm-discrete-II}(c) using the  notation $\{f(x,y)\}$,
where $f(x,y)$ is an inequality in the variables $x$ and $y$,
to denote the set of all points $(x,y)\in \Z^2$
satisfying the inequality $f(x,y)$.
Thus $\{x\leq y\}=\{(x,y)\in \Z^2: x\leq y\}$ for instance.

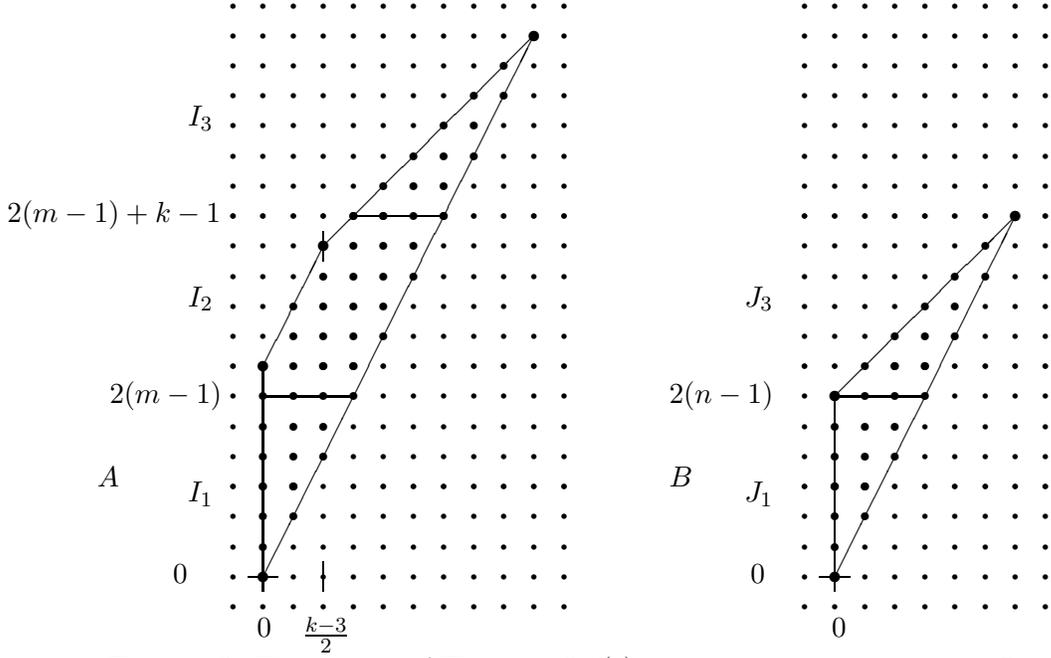
\begin{figure}[ht]\label{case (c)}
\setlength{\unitlength}{4mm}
\begin{center}
\begin{picture}(20,19)

\put(0,0){\line(0,1){7}} \put(0,7){\line(1,2){2}}
\put(2,11){\line(1,1){7}} \put(0,0){\line(1,2){9}}

\put(-5.5,3){$A$} \put(13.5,3){$B$}

\put(-.5,0){\line(1,0){1}} \put(18.5,0){\line(1,0){1}}

\put(19,0){\line(0,1){6}} \put(19,6){\line(1,1){6}}
\put(19,0){\line(1,2){6}}

\put(0,.5){\line(0,-1){1}} \put(-.2,-2){$0$} \put(18.9,-2){$0$}

\put(-2.5,9){$I_2$} \put(-2.5,2.5){$I_1$} \put(-2.5,15){$I_3$}

\put(16,9){$J_3$} \put(16,2.5){$J_1$}

\put(0,6){\line(1,0){3}} \put(19,6){\line(1,0){3}}

\put(3,12){\line(1,0){3}} \put(-5.1,5.8){$2(m-1)$}
\put(13.5,5.8){$2(n-1)$}

\put(-8.5,11.8){$2(m-1)+k-1$}

\put(19,.5){\line(0,-1){1}}

\put(16.2,-.2){$0$} \put(-3,-.2){$0$}
\put(1.3,-2.2){$\frac{k-3}{2}$} \put(2,11.5){\line(0,-1){1}}
\put(2,.5){\line(0,-1){1}}

\put(0,0){\circle*{0.35}}\put(0,1){\circle*{0.25}}\put(0,2){\circle*{0.25}}\put(0,3){\circle*{0.25}}\put(0,3){\circle*{0.25}}\put(0,4){\circle*{0.25}}
\put(0,5){\circle*{0.25}}\put(0,6){\circle*{0.25}}\put(0,7){\circle*{0.35}}\put(0,7){\circle*{0.2}}\put(0,8){\circle*{0.2}}\put(0,9){\circle*{0.2}}
\put(0,10){\circle*{0.2}}\put(0,11){\circle*{0.2}}\put(0,12){\circle*{0.2}}\put(0,12){\circle*{0.2}}\put(0,13){\circle*{0.2}}\put(0,14){\circle*{0.2}}
\put(0,15){\circle*{0.2}}\put(0,16){\circle*{0.2}}\put(0,17){\circle*{0.2}}\put(0,18){\circle*{0.2}}\put(0,19){\circle*{0.2}}\put(0,-1){\circle*{0.2}}

\put(-1,0){\circle*{0.2}}\put(-1,1){\circle*{0.2}}\put(-1,2){\circle*{0.2}}\put(-1,3){\circle*{0.2}}\put(-1,3){\circle*{0.2}}\put(-1,4){\circle*{0.2}}
\put(-1,5){\circle*{0.2}}\put(-1,6){\circle*{0.2}}\put(-1,7){\circle*{0.2}}\put(-1,7){\circle*{0.2}}\put(-1,8){\circle*{0.2}}\put(-1,9){\circle*{0.2}}
\put(-1,10){\circle*{0.2}}\put(-1,11){\circle*{0.2}}\put(-1,12){\circle*{0.2}}\put(-1,12){\circle*{0.2}}\put(-1,13){\circle*{0.2}}\put(-1,14){\circle*{0.2}}
\put(-1,15){\circle*{0.2}}\put(-1,16){\circle*{0.2}}\put(-1,17){\circle*{0.2}}\put(-1,18){\circle*{0.2}}\put(-1,19){\circle*{0.2}}\put(-1,-1){\circle*{0.2}}

\put(1,0){\circle*{0.2}}\put(1,1){\circle*{0.2}}\put(1,2){\circle*{0.25}}\put(1,3){\circle*{0.25}}\put(1,3){\circle*{0.25}}\put(1,4){\circle*{0.25}}
\put(1,5){\circle*{0.25}}\put(1,6){\circle*{0.25}}\put(1,7){\circle*{0.25}}\put(1,8){\circle*{0.25}}\put(1,9){\circle*{0.25}}
\put(1,10){\circle*{0.2}}\put(1,11){\circle*{0.2}}\put(1,12){\circle*{0.2}}\put(1,12){\circle*{0.2}}\put(1,13){\circle*{0.2}}\put(1,14){\circle*{0.2}}
\put(1,15){\circle*{0.2}}\put(1,16){\circle*{0.2}}\put(1,17){\circle*{0.2}}\put(1,18){\circle*{0.2}}\put(1,19){\circle*{0.2}}\put(1,-1){\circle*{0.2}}

\put(2,0){\circle*{0.2}}\put(2,1){\circle*{0.2}}\put(2,2){\circle*{0.2}}\put(2,3){\circle*{0.2}}\put(2,4){\circle*{0.25}}
\put(2,5){\circle*{0.25}}\put(2,6){\circle*{0.25}}\put(2,7){\circle*{0.25}}\put(2,7){\circle*{0.25}}\put(2,8){\circle*{0.25}}\put(2,9){\circle*{0.25}}
\put(2,10){\circle*{0.25}}\put(2,11){\circle*{0.35}}\put(2,12){\circle*{0.2}}\put(2,12){\circle*{0.2}}\put(2,13){\circle*{0.2}}\put(2,14){\circle*{0.2}}
\put(2,15){\circle*{0.2}}\put(2,16){\circle*{0.2}}\put(2,17){\circle*{0.2}}\put(2,18){\circle*{0.2}}\put(2,19){\circle*{0.2}}\put(2,-1){\circle*{0.2}}

\put(3,0){\circle*{0.2}}\put(3,1){\circle*{0.2}}\put(3,2){\circle*{0.2}}\put(3,3){\circle*{0.2}}\put(3,3){\circle*{0.2}}\put(3,4){\circle*{0.2}}
\put(3,5){\circle*{0.2}}\put(3,6){\circle*{0.25}}\put(3,7){\circle*{0.25}}\put(3,7){\circle*{0.25}}\put(3,8){\circle*{0.25}}\put(3,9){\circle*{0.25}}
\put(3,10){\circle*{0.25}}\put(3,11){\circle*{0.25}}\put(3,12){\circle*{0.25}}\put(3,12){\circle*{0.25}}\put(3,13){\circle*{0.2}}\put(3,14){\circle*{0.2}}
\put(3,15){\circle*{0.2}}\put(3,16){\circle*{0.2}}\put(3,17){\circle*{0.2}}\put(3,18){\circle*{0.2}}\put(3,19){\circle*{0.2}}\put(3,-1){\circle*{0.2}}

\put(4,0){\circle*{0.2}}\put(4,1){\circle*{0.2}}\put(4,2){\circle*{0.2}}\put(4,3){\circle*{0.2}}\put(4,3){\circle*{0.2}}\put(4,4){\circle*{0.2}}
\put(4,5){\circle*{0.2}}\put(4,6){\circle*{0.2}}\put(4,7){\circle*{0.2}}\put(4,7){\circle*{0.2}}\put(4,8){\circle*{0.25}}\put(4,9){\circle*{0.25}}
\put(4,10){\circle*{0.25}}\put(4,11){\circle*{0.25}}\put(4,12){\circle*{0.25}}\put(4,12){\circle*{0.25}}\put(4,13){\circle*{0.25}}\put(4,14){\circle*{0.2}}
\put(4,15){\circle*{0.2}}\put(4,16){\circle*{0.2}}\put(4,17){\circle*{0.2}}\put(4,18){\circle*{0.2}}\put(4,19){\circle*{0.2}}\put(4,-1){\circle*{0.2}}

\put(5,0){\circle*{0.2}}\put(5,1){\circle*{0.2}}\put(5,2){\circle*{0.2}}\put(5,3){\circle*{0.2}}\put(5,3){\circle*{0.2}}\put(5,4){\circle*{0.2}}
\put(5,5){\circle*{0.2}}\put(5,6){\circle*{0.2}}\put(5,7){\circle*{0.2}}\put(5,7){\circle*{0.2}}\put(5,8){\circle*{0.2}}\put(5,9){\circle*{0.2}}
\put(5,10){\circle*{0.25}}\put(5,11){\circle*{0.25}}\put(5,12){\circle*{0.25}}\put(5,13){\circle*{0.25}}\put(5,14){\circle*{0.25}}
\put(5,15){\circle*{0.2}}\put(5,16){\circle*{0.2}}\put(5,17){\circle*{0.2}}\put(5,18){\circle*{0.2}}\put(5,19){\circle*{0.2}}\put(5,-1){\circle*{0.2}}

\put(6,0){\circle*{0.2}}\put(6,1){\circle*{0.2}}\put(6,2){\circle*{0.2}}\put(6,3){\circle*{0.2}}\put(6,3){\circle*{0.2}}\put(6,4){\circle*{0.2}}
\put(6,5){\circle*{0.2}}\put(6,6){\circle*{0.2}}\put(6,7){\circle*{0.2}}\put(6,7){\circle*{0.2}}\put(6,8){\circle*{0.2}}\put(6,9){\circle*{0.2}}
\put(6,10){\circle*{0.2}}\put(6,11){\circle*{0.2}}\put(6,12){\circle*{0.25}}\put(6,12){\circle*{0.25}}\put(6,13){\circle*{0.25}}\put(6,14){\circle*{0.25}}
\put(6,15){\circle*{0.25}}\put(6,16){\circle*{0.2}}\put(6,17){\circle*{0.2}}\put(6,18){\circle*{0.2}}\put(6,19){\circle*{0.2}}\put(6,-1){\circle*{0.2}}

\put(7,0){\circle*{0.2}}\put(7,1){\circle*{0.2}}\put(7,2){\circle*{0.2}}\put(7,3){\circle*{0.2}}\put(7,3){\circle*{0.2}}\put(7,4){\circle*{0.2}}
\put(7,5){\circle*{0.2}}\put(7,6){\circle*{0.2}}\put(7,7){\circle*{0.2}}\put(7,7){\circle*{0.2}}\put(7,8){\circle*{0.2}}\put(7,9){\circle*{0.2}}
\put(7,10){\circle*{0.2}}\put(7,11){\circle*{0.2}}\put(7,12){\circle*{0.2}}\put(7,12){\circle*{0.2}}\put(7,13){\circle*{0.2}}\put(7,14){\circle*{0.25}}
\put(7,15){\circle*{0.25}}\put(7,16){\circle*{0.25}}\put(7,17){\circle*{0.2}}\put(7,18){\circle*{0.2}}\put(7,19){\circle*{0.2}}\put(7,-1){\circle*{0.2}}

\put(8,0){\circle*{0.2}}\put(8,1){\circle*{0.2}}\put(8,2){\circle*{0.2}}\put(8,3){\circle*{0.2}}\put(8,3){\circle*{0.2}}\put(8,4){\circle*{0.2}}
\put(8,5){\circle*{0.2}}\put(8,6){\circle*{0.2}}\put(8,7){\circle*{0.2}}\put(8,7){\circle*{0.2}}\put(8,8){\circle*{0.2}}\put(8,9){\circle*{0.2}}
\put(8,10){\circle*{0.2}}\put(8,11){\circle*{0.2}}\put(8,12){\circle*{0.2}}\put(8,12){\circle*{0.2}}\put(8,13){\circle*{0.2}}\put(8,14){\circle*{0.2}}
\put(8,15){\circle*{0.2}}\put(8,16){\circle*{0.25}}\put(8,17){\circle*{0.25}}\put(8,18){\circle*{0.2}}\put(8,19){\circle*{0.2}}\put(8,-1){\circle*{0.2}}

\put(9,0){\circle*{0.2}}\put(9,1){\circle*{0.2}}\put(9,2){\circle*{0.2}}\put(9,3){\circle*{0.2}}\put(9,3){\circle*{0.2}}\put(9,4){\circle*{0.2}}
\put(9,5){\circle*{0.2}}\put(9,6){\circle*{0.2}}\put(9,7){\circle*{0.2}}\put(9,7){\circle*{0.2}}\put(9,8){\circle*{0.2}}\put(9,9){\circle*{0.2}}
\put(9,10){\circle*{0.2}}\put(9,11){\circle*{0.2}}\put(9,12){\circle*{0.2}}\put(9,12){\circle*{0.2}}\put(9,13){\circle*{0.2}}\put(9,14){\circle*{0.2}}
\put(9,15){\circle*{0.2}}\put(9,16){\circle*{0.2}}\put(9,17){\circle*{0.2}}\put(9,18){\circle*{0.35}}\put(9,19){\circle*{0.2}}\put(9,-1){\circle*{0.2}}

\put(10,0){\circle*{0.2}}\put(10,1){\circle*{0.2}}\put(10,2){\circle*{0.2}}\put(10,3){\circle*{0.2}}\put(10,3){\circle*{0.2}}\put(10,4){\circle*{0.2}}
\put(10,5){\circle*{0.2}}\put(10,6){\circle*{0.2}}\put(10,7){\circle*{0.2}}\put(10,7){\circle*{0.2}}\put(10,8){\circle*{0.2}}\put(10,9){\circle*{0.2}}
\put(10,10){\circle*{0.2}}\put(10,11){\circle*{0.2}}\put(10,12){\circle*{0.2}}\put(10,12){\circle*{0.2}}\put(10,13){\circle*{0.2}}\put(10,14){\circle*{0.2}}
\put(10,15){\circle*{0.2}}\put(10,16){\circle*{0.2}}\put(10,17){\circle*{0.2}}\put(10,18){\circle*{0.2}}\put(10,19){\circle*{0.2}}\put(10,-1){\circle*{0.2}}

\put(18,0){\circle*{0.2}}\put(18,1){\circle*{0.2}}\put(18,2){\circle*{0.2}}\put(18,3){\circle*{0.2}}\put(18,3){\circle*{0.2}}\put(18,4){\circle*{0.2}}
\put(18,5){\circle*{0.2}}\put(18,6){\circle*{0.2}}\put(18,7){\circle*{0.2}}\put(18,7){\circle*{0.2}}\put(18,8){\circle*{0.2}}\put(18,9){\circle*{0.2}}
\put(18,10){\circle*{0.2}}\put(18,11){\circle*{0.2}}\put(18,12){\circle*{0.2}}\put(18,12){\circle*{0.2}}\put(18,13){\circle*{0.2}}\put(18,14){\circle*{0.2}}
\put(18,15){\circle*{0.2}}\put(18,16){\circle*{0.2}}\put(18,17){\circle*{0.2}}\put(18,18){\circle*{0.2}}\put(18,19){\circle*{0.2}}\put(18,-1){\circle*{0.2}}

\put(19,0){\circle*{0.35}}\put(19,1){\circle*{0.25}}\put(19,2){\circle*{0.25}}\put(19,3){\circle*{0.25}}\put(19,3){\circle*{0.25}}\put(19,4){\circle*{0.25}}
\put(19,5){\circle*{0.25}}\put(19,6){\circle*{0.35}}\put(19,7){\circle*{0.2}}\put(19,7){\circle*{0.2}}\put(19,8){\circle*{0.2}}\put(19,9){\circle*{0.2}}
\put(19,10){\circle*{0.2}}\put(19,11){\circle*{0.2}}\put(19,12){\circle*{0.2}}\put(19,12){\circle*{0.2}}\put(19,13){\circle*{0.2}}\put(19,14){\circle*{0.2}}
\put(19,15){\circle*{0.2}}\put(19,16){\circle*{0.2}}\put(19,17){\circle*{0.2}}\put(19,18){\circle*{0.2}}\put(19,19){\circle*{0.2}}\put(19,-1){\circle*{0.2}}

\put(20,0){\circle*{0.2}}\put(20,1){\circle*{0.2}}\put(20,2){\circle*{0.25}}\put(20,3){\circle*{0.25}}\put(20,3){\circle*{0.25}}\put(20,4){\circle*{0.25}}
\put(20,5){\circle*{0.25}}\put(20,6){\circle*{0.25}}\put(20,7){\circle*{0.25}}\put(20,7){\circle*{0.2}}\put(20,8){\circle*{0.2}}\put(20,9){\circle*{0.2}}
\put(20,10){\circle*{0.2}}\put(20,11){\circle*{0.2}}\put(20,12){\circle*{0.2}}\put(20,12){\circle*{0.2}}\put(20,13){\circle*{0.2}}\put(20,14){\circle*{0.2}}
\put(20,15){\circle*{0.2}}\put(20,16){\circle*{0.2}}\put(20,17){\circle*{0.2}}\put(20,18){\circle*{0.2}}\put(20,19){\circle*{0.2}}\put(20,-1){\circle*{0.2}}

\put(21,0){\circle*{0.2}}\put(21,1){\circle*{0.2}}\put(21,2){\circle*{0.2}}\put(21,3){\circle*{0.2}}\put(21,3){\circle*{0.2}}\put(21,4){\circle*{0.25}}
\put(21,5){\circle*{0.25}}\put(21,6){\circle*{0.25}}\put(21,7){\circle*{0.25}}\put(21,7){\circle*{0.25}}\put(21,8){\circle*{0.25}}\put(21,9){\circle*{0.2}}
\put(21,10){\circle*{0.2}}\put(21,11){\circle*{0.2}}\put(21,12){\circle*{0.2}}\put(21,12){\circle*{0.2}}\put(21,13){\circle*{0.2}}\put(21,14){\circle*{0.2}}
\put(21,15){\circle*{0.2}}\put(21,16){\circle*{0.2}}\put(21,17){\circle*{0.2}}\put(21,18){\circle*{0.2}}\put(21,19){\circle*{0.2}}\put(21,-1){\circle*{0.2}}

\put(22,0){\circle*{0.2}}\put(22,1){\circle*{0.2}}\put(22,2){\circle*{0.2}}\put(22,3){\circle*{0.2}}\put(22,3){\circle*{0.2}}\put(22,4){\circle*{0.2}}
\put(22,5){\circle*{0.2}}\put(22,6){\circle*{0.25}}\put(22,7){\circle*{0.25}}\put(22,7){\circle*{0.25}}\put(22,8){\circle*{0.25}}\put(22,9){\circle*{0.25}}
\put(22,10){\circle*{0.2}}\put(22,11){\circle*{0.2}}\put(22,12){\circle*{0.2}}\put(22,12){\circle*{0.2}}\put(22,13){\circle*{0.2}}\put(22,14){\circle*{0.2}}
\put(22,15){\circle*{0.2}}\put(22,16){\circle*{0.2}}\put(22,17){\circle*{0.2}}\put(22,18){\circle*{0.2}}\put(22,19){\circle*{0.2}}\put(22,-1){\circle*{0.2}}

\put(23,0){\circle*{0.2}}\put(23,1){\circle*{0.2}}\put(23,2){\circle*{0.2}}\put(23,3){\circle*{0.2}}\put(23,3){\circle*{0.2}}\put(23,4){\circle*{0.2}}
\put(23,5){\circle*{0.2}}\put(23,6){\circle*{0.2}}\put(23,7){\circle*{0.2}}\put(23,7){\circle*{0.2}}\put(23,8){\circle*{0.25}}\put(23,9){\circle*{0.25}}
\put(23,10){\circle*{0.25}}\put(23,11){\circle*{0.2}}\put(23,12){\circle*{0.2}}\put(23,12){\circle*{0.2}}\put(23,13){\circle*{0.2}}\put(23,14){\circle*{0.2}}
\put(23,15){\circle*{0.2}}\put(23,16){\circle*{0.2}}\put(23,17){\circle*{0.2}}\put(23,18){\circle*{0.2}}\put(23,19){\circle*{0.2}}\put(23,-1){\circle*{0.2}}

\put(24,0){\circle*{0.2}}\put(24,1){\circle*{0.2}}\put(24,2){\circle*{0.2}}\put(24,3){\circle*{0.2}}\put(24,3){\circle*{0.2}}\put(24,4){\circle*{0.2}}
\put(24,5){\circle*{0.2}}\put(24,6){\circle*{0.2}}\put(24,7){\circle*{0.2}}\put(24,7){\circle*{0.2}}\put(24,8){\circle*{0.2}}\put(24,9){\circle*{0.2}}
\put(24,10){\circle*{0.25}}\put(24,11){\circle*{0.25}}\put(24,12){\circle*{0.2}}\put(24,12){\circle*{0.2}}\put(24,13){\circle*{0.2}}\put(24,14){\circle*{0.2}}
\put(24,15){\circle*{0.2}}\put(24,16){\circle*{0.2}}\put(24,17){\circle*{0.2}}\put(24,18){\circle*{0.2}}\put(24,19){\circle*{0.2}}\put(24,-1){\circle*{0.2}}

\put(25,0){\circle*{0.2}}\put(25,1){\circle*{0.2}}\put(25,2){\circle*{0.2}}\put(25,3){\circle*{0.2}}\put(25,3){\circle*{0.2}}\put(25,4){\circle*{0.2}}
\put(25,5){\circle*{0.2}}\put(25,6){\circle*{0.2}}\put(25,7){\circle*{0.2}}\put(25,7){\circle*{0.2}}\put(25,8){\circle*{0.2}}\put(25,9){\circle*{0.2}}
\put(25,10){\circle*{0.2}}\put(25,11){\circle*{0.2}}\put(25,12){\circle*{0.35}}\put(25,12){\circle*{0.2}}\put(25,13){\circle*{0.2}}\put(25,14){\circle*{0.2}}
\put(25,15){\circle*{0.2}}\put(25,16){\circle*{0.2}}\put(25,17){\circle*{0.2}}\put(25,18){\circle*{0.2}}\put(25,19){\circle*{0.2}}\put(25,-1){\circle*{0.2}}

\put(26,0){\circle*{0.2}}\put(26,1){\circle*{0.2}}\put(26,2){\circle*{0.2}}\put(26,3){\circle*{0.2}}\put(26,3){\circle*{0.2}}\put(26,4){\circle*{0.2}}
\put(26,5){\circle*{0.2}}\put(26,6){\circle*{0.2}}\put(26,7){\circle*{0.2}}\put(26,7){\circle*{0.2}}\put(26,8){\circle*{0.2}}\put(26,9){\circle*{0.2}}
\put(26,10){\circle*{0.2}}\put(26,11){\circle*{0.2}}\put(26,12){\circle*{0.2}}\put(26,12){\circle*{0.2}}\put(26,13){\circle*{0.2}}\put(26,14){\circle*{0.2}}
\put(26,15){\circle*{0.2}}\put(26,16){\circle*{0.2}}\put(26,17){\circle*{0.2}}\put(26,18){\circle*{0.2}}\put(26,19){\circle*{0.2}}\put(26,-1){\circle*{0.2}}

\end{picture}
\end{center}
\vspace{3mm} \caption{Illustration of Theorem
\ref{thm-discrete-II}(c) with $m=4$, $n=4$ and $k=7.$}
\end{figure}

\begin{theorem}\label{thm-discrete-II} Let $A,\,B\subseteq \R^2$ be
finite two-dimensional subsets with $0\in A\cap B$.  Let $m$ be the
maximal number of points in $A$ contained on a horizontal line and
let $n$ be the maximal number of points in $B$ contained on a
horizontal line. Suppose $m,\,n\geq 2$. Then,
\begin{equation}\label{eq:eq-discreteII}
|A+B|\ge \left(\frac{|A|}{m}+\frac{|B|}{n}-1\right)(m+n-1).
\end{equation}
Moreover, if equality holds in \eqref{eq:eq-discreteII}, then, up
to a linear transformation  of the form $(x,y)\mapsto(\alpha
x+\gamma y,\beta y)$,
 where $\alpha,\beta,\gamma\in \R$ and $\alpha$ and $\beta$ non-zero, one of the following holds:
\begin{itemize}
\item[(a)]
$A$ and $B$ are standard trapezoids $T(m,h,c,d)$ and $T(n,h',c,d)$
with common slopes $c$ and $d$.
\item[(b)] $A$ is an $\epsilon$--standard trapezoid $T_{\epsilon}(m,h,c,d)$ and $B$ is
a standard trapezoid $T(n, h', c,d)$ with $h'=(n-1)d+1$, or the same holds with the roles of
$A$ and $B$ (and $m$ and $n$)
reversed.
\item[(c)]  Up to translation, 
$$
A=\{x\geq 0 \}\cap  \{y\geq 2x\} \cap  \{ y\leq
x+2m+\frac{k-5}{2}\} \cap   \{ y\leq 2x+2m-1 \}
$$
and
$$
B=\{x\geq 0 \}\cap  \{y\geq 2x\} \cap  \{ y\leq x+2n-2\},\nn
$$
where $k\in \mathbb N$ is odd, or the same holds with the roles of
$A$ and $B$ (and $m$ and $n$) reversed.
\end{itemize}
\end{theorem}

\begin{proof}
For a finite set $X\subseteq \R^2$ and $i\in \R$, we let $X_i=X\cap
\{(x,y)\in \R^2\mid y=i\}$ denote the intersection of $X$ with the
horizontal line defined by $y=i$. Recall that $\pi':\R^2\rightarrow
\R$ denotes the horizontal projection map onto the vertical axis.
Since $A$ and $B$ are two-dimensional, we have $|\pi'(A)|\geq 2$ and
$|\pi'(B)|\geq 2$.

For $X\subseteq \R^2$, we let $\mathsf c(X)\subseteq \R^2$ denote the
subset with $\pi'(X)=\pi'(\mathsf c(X))$ such that $\mathsf
c(X)_i=\{(0,i),(1,i),\ldots,(|X_i|-1,i)\}$ for $i\in \pi'(X)$. Then
$\mathsf c(X)$ is the {\it horizontal compression} of $X$.
The
following properties are easily observed regarding $\mathsf c(X)$:
$$
|\mathsf c (X)|=|X|,\quad |\pi(\mathsf c (X))|=\max_i |X_i|\;\und\; \pi' (\mathsf c (X))=\pi'(X).
$$
Additionally, we have  \ber \nn  |A+B|&=& \Summ{t\in
\R}\,|\bigcup_{i+j=t}(A_i+B_j)|\\\label{B1} &\geq& \Summ{t\in
\R}\max_{i+j=t}\{|A_i+B_j|\}\\\label{B2}&\geq& \Summ{t\in
\R}\max_{i+j=t}\{|A_i|+|B_j|-1\} \\&=&\label{B3} \Summ{t\in
\R}\,|\bigcup_{i+j=t}(\mathsf c(A)_i+\mathsf c(B)_j)|= |\mathsf
c(A)+\mathsf c(B)|, \eer where \eqref{B2} follows from Theorem \ref{CDT-forZ}.

Since $|\pi(\mathsf c (A))|=\underset{i}{\max}\, |A_i|=m$ and $|\pi(\mathsf c (B))|=\underset{i}{\max }\,|B_i|=n$,
inequality \eqref{eq:eq-discreteII} now follows from
\eqref{eq:tightdiscrete}.
Moreover, if equality holds in \eqref{eq:eq-discreteII}, then $\mathsf
c(A)$ and $\mathsf c(B)$ are an extremal pair for \eqref{eq:tightdiscrete}. Thus,  in view Theorem \ref{thm:invdisc}, by applying an appropriate  linear transformation of the form $(x,y)\mapsto (x,\beta^{-1}y)$, we can w.l.o.g. re-scale the vertical axis so that $\mathsf c(A)$ and $\mathsf c(B)$
are   standard trapezoids $T(m,h,c,d)$ and $T(n,h',c,d)$ with common
slopes $c$ and $d$. Moreover, since the horizontal sections of $\mathsf c(A)$ are arithmetic progressions with difference $(1,0)$, and since,  by definition of a standard trapezoid, the vertical sections of $\mathsf c(A)$ are arithmetic progressions with difference $(0,1)$, it follows that $\mathsf c(A)\subseteq \Z^2$, and likewise $\mathsf c(B)\subseteq \Z^2$, whence $c,\,d\in\Z$.

Since $\mathsf c(A)$ and $\mathsf c(B)$ are horizontally compressed,
we see that $$c\leq 0\und d\geq 0,$$ so that
  $h=|\pi'(A)|$ and $h'=|\pi'(B)|$, and   by
translation we assume $\pi'(A)= \{0,1,\ldots,h-1\}$
and $\pi'(B)=\{0,1,\ldots,h'-1\}$. This assumption will remain in place for the remainder of the proof, and all affine transformations employed in the proof will preserve this assumption.

Let \begin{align}\nn& I_1=[0,(m-1)d], \; I_2=[(m-1)d, h-1+(m-1)c],\;
I_3=[h-1+(m-1)c ,h-1],\\& J_1=[0,(n-1)d],\; J_2=[(n-1)d,
h'-1+(n-1)c],\; J_3=[h'-1+(n-1)c ,h'-1].\nn\end{align} Then
\begin{equation}\label{eq:cardinalities}
|A_i|=\left\{
\begin{array}{ll}
 \lfloor\frac{i}{d}\rfloor+1, & i\in I_1 \\
  m, & i\in I_2 \\
  \lfloor\frac{-i+h-1}{-c}\rfloor+1, & i\in I_3
  \end{array}
  \right.\;
  \quad\und\quad
|B_j|=\left\{
\begin{array}{ll}
\lfloor\frac{j}{d}\rfloor+1, & j\in J_1 \\
n, & j\in J_2 \\
\lfloor\frac{-j+h'-1}{-c}\rfloor+1, & j\in J_3.
\end{array}
\right. \end{equation}

Let us call a pair $(k,l)\in [0,h-1]\times [0,h'-1]$  {\it tight} if it attains the maximum in
\eqref{B2} for $t=k+l$. Since equality holds in \eqref{B1} and \eqref{B2}, it
follows, for $t\in \{0,1,\ldots,h+h'-2\}$, that
\begin{align}\label{kl-card}&|A_k+B_l|=|A_k|+|B_l|-1\und
\\\label{kl-subset} &\bigcup_{i+j=t}(A_{i}+B_j)=
A_k+A_l,\end{align} for any tight pair $(k,l)$ with $k+l=t$.
In view
of \eqref{eq:cardinalities}, if $(k,l)\in [0,md-1]\cap (I_1\cup I_2)\times [0,nd-1]\cap (J_1\cup J_2)$ is a pair with $k+l\equiv \theta\mod d$ and $k,\,l\in [0,\theta]+d\Z$, where $\theta\in [0,d-1]$, then $(k,l)$ is tight.
 We refer to this as the {\it modular rule} for tightness. It will be most often applied in the case of pairs $(k,l)\in I_1\times J_1\subseteq [0,md-1]\cap (I_1\cup I_2)\times [0,nd-1]\cap (J_1\cup J_2)$. In particular, a pair $(k,l)\in I_1\times J_1$ is tight whenever
$k\equiv 0 \mod d$ or $l\equiv 0\mod d$. A similar statement holds for pairs
 $(k,l)\in I_3\times J_3$, though, by symmetry, we will have little need of an explicit formulation for these cases.  Also, every pair
$(k,l)\in I_2\times J_2$ is tight.

Our next goal is to show each $A_i$ and $B_j$ is an arithmetic progression with common difference $(\alpha,0)$, for some $\alpha>0$.
First consider $i\in I_2$ and $j\in J_2$. Then $|A_i|=m\geq 2$ and $|B_j|=n\geq 2$. Consequently, in view of \eqref{kl-card} and Theorem \ref{CDT-forZ}, it follows that all $A_i$ with $i\in I_2$ and $B_j$ with $j\in J_2$ are arithmetic progressions
of common difference (say) $(\alpha,0)$, where w.l.o.g. $\alpha>0$. Now consider $i\in I_1$. If $|A_i|=1$, then $A_i$ is trivially an arithmetic progression with any difference, so assume $|A_i|\geq 2$. Then the pair $(i,(n-1)d)\in I_1\times J_1$ is tight by the modular rule for tightness, whence Theorem \ref{CDT-forZ} and \eqref{kl-card}  show $A_i$ is also an arithmetic progression with difference $(\alpha,0)$. The same argument works for $j\in J_1$, and symmetrical arguments handle the cases $i\in I_3$ and $j\in J_3$, so all $A_i$ and $B_j$ are arithmetic progressions with common difference $(\alpha,0)$, as claimed. By re-scaling the horizontal axis as need be, we may w.l.o.g. assume $\alpha=1$.

Let $a_i,\,a'_i\in \Z$ be, respectively, the minimal and maximal
first coordinate of an element from $A_i$ and let $b_j,\,b'_j\in \Z$
be, respectively, the minimal and maximal first coordinate of an
element from $B_j$, for $i\in [0,h-1]$ and $j\in [0,h'-1]$.

We first consider the case $|I_3|=|J_3|=1$, that is, $c=0$.

\begin{claim}\label{claim:c=0} Suppose that $c=0$. Then, up to an
appropriate linear transformation of the form $(x,y)\mapsto
(x-\alpha y ,y)$, one of the following conditions holds:

{\rm (i)}   $A$ and $B$ are standard trapezoids with common slopes $0$ and
$d$, or

{\rm (ii)} $A$ is an $\epsilon$--standard trapezoid
$T_{\epsilon}(m,h,0,\pm d)$ and $B$ is a standard trapezoid
$T(n,(n-1)d+1,0,\pm d)$ with common slopes $0$ and $\pm d$, or the same holds with the roles of $A$ and $B$ (and $m$ and $n$) reversed.
\end{claim}

\begin{proof} As both descriptions (i) and (ii) are invariant of translation and horizontal reflection, we can translate $A$ and $B$ as need be, and it suffices to prove the theorem for the horizontal reflections of $A$ and $B$. Thus we first show that, by applying an appropriate linear transformation of the form $(x,y)\mapsto
(\pm (x-\alpha y) ,y)$ and translating, we may assume that
\begin{equation}\label{eq:aligned}
a_i=b_j=0 \quad \mbox{ for all } i\in I_1\und j\in J_1.
\end{equation}
Note that if the coefficient of $x$ is $-1$, then the affine transformation involves a horizontal reflection which swaps the roles of the $a'_i$ and $a_i$, etc. Thus \eqref{eq:aligned} is equivalent to saying that either the sequences $a_i$ and $b_i$ or the sequences $a'_i$ and $b'_i$ are arithmetic progressions with common difference $\alpha\in \R$.

Suppose that $d\ge 1$, since otherwise \eqref{eq:aligned} is clear. By the modular rule and  \eqref{kl-subset}, it follows, for all $i,\,i'\in [0,m-1]$ and
$j,\,j'\in [0,n-1]$ with $i+j=i'+j'$, that
$$
a_{id}+b_{jd}=a_{i'd}+b_{j'd}.
$$
Hence the sequences $(a_{id}:\,
i\in [0,m-1])$ and $(b_{jd}:\, j\in [0,n-1])$ are arithmetic
progressions with the same common difference in view of $m,\,n\geq 2$. By an appropriate
transformation of the form $(x,y)\mapsto (x-\alpha
y+\alpha',y)$, we may assume that the difference is zero and thus
\be\label{viiteet} a_{id}=b_{jd}=0\quad \mbox{  for all }\;i\in [0,m-1] \und j\in [0,n-1].\ee
Moreover, the sequences $(a_{id}':\, i\in [0,m-1])$ and $(b_{jd}':\,
j\in [0,n-1])$ are then arithmetic progressions with  difference one. Now
\eqref{eq:aligned} holds if $d=1$.

Suppose that $d\ge 2$. By the modular rule and \eqref{kl-subset}, we have
\be\label{veetiit}
A_{0}+B_{j}\supseteq A_1+B_{j-1}\quad \mbox{ for } \;j\in J_1\setminus \{0\}.
\ee
Moreover, if $j\not\equiv 0\mod d$, then the pair $(1,j-1)$ is also tight by the modular rule, whence \eqref{kl-subset} shows equality must hold in \eqref{veetiit}. Hence $a_0+b_j=a_1+b_{j-1}$ for $J_1\setminus \{0\}$ with $j\not\equiv 0\mod d$, so that the difference $b_j-b_{j-1}$ is constant,
being equal to $\delta:=a_1-a_0$. For $J_1\setminus \{0\}$ with $j\equiv 0\mod d$, we note that
$A_1+B_{j-1}$ is an
arithmetic progression whose length is only one  less than the arithmetic
progression $A_0+B_{j}$, whence \eqref{veetiit} instead implies $a_1+b_{j-1}\in a_0+b_{j}+\{0,1\}$, so that $b_j-b_{j-1}\in \{\delta,\delta-1\}$ for $J_1\setminus \{0\}$ with  $j\equiv 0\mod d$.

 Repeating these arguments with the roles of $A$ and $B$
swapped, we likewise conclude (recall $d\geq 2$) that
$a_i-a_{i-1}=b_1-b_0=a_1-a_0=\delta$ for $i\not\equiv 0\mod d$,
while $a_i-a_{i-1}\in \{\delta,\delta-1\}$ for $i\equiv 0\mod d$.
The special case with $j=d$ gives, in view of \eqref{viiteet},
$$d\delta=(a_0+\delta)+(b_0+(d-1)\delta)=a_1+b_{d-1}\in a_0+b_d+ \{0,1\}=\{0,1\}.$$
Thus $\delta=0$ or $\delta=\frac{1}{d}$. If $\delta=0$, then
$(a_i:\,i\in I_1)$ and $(b_i:\,i\in J_1)$ are arithmetic progressions of
common difference $0$. If $\delta=\frac{1}{d}$, then
$a'_i=\frac{i}{d}$ and $b'_j=\frac{j}{d}$ for $i\in I_1$ and $j\in
J_1$, whence $(a'_i:\,i\in I_1)$ and $(b'_i:\,i\in J_1)$ are arithmetic
progressions with common difference $\frac{1}{d}$. In either case,
\eqref{eq:aligned} is established as previously explained. We now assume \eqref{eq:aligned} holds for $A$ and $B$ and consider three cases.

Moreover, and this is important for later applications of the claim, when $|J_2|=1$, we show that the conclusion of the claim holds for  {\it any} transformation of the given form for which \eqref{eq:aligned} holds, and thus also for the horizontal reflection of such a transform.

\subsection*{Case 1: }  $|I_2|=|J_2|=1$. Then $A$ and $B$ are the standard
trapezoids $T(m,(m-1)d+1,0,d)$ and $T(n,(n-1)d+1,0,d)$ respectively,
giving part (i) of the claim.

\subsection*{Case 2: }   $|I_2|\ge 2$ and $|J_2|\ge 2$.
Since all pairs $(i,j)\in I_2\times J_2$ are tight, it follows from \eqref{kl-subset} that $A_i+A_j=A_{i'}+A_{j'}$ for all $i,\,i'\in I_2$ and $j,\,j'\in J_2$. Thus, in view of  $|I_2|\ge 2$ and $|J_2|\ge 2$, it follows that the sequences $(a_i:\, i\in
I_2)$ and  $(b_j:\, j\in J_2)$ are arithmetic progressions with the same
common difference (say) $\alpha$.
Thus, if $d=0$, then  $|I_1|=|J_1|=1$ and Claim 1(i) follows by applying a linear transformation of the form $(x,y)\mapsto
(x-\alpha y ,y)$.

So consider the case  $d\ge 1$. Then $0,1\in I_1$, whence the modular rule implies the pair $(1,(n-1)d)$ is tight, in which case  \eqref{kl-subset} implies
\be\label{stable}
A_0+B_{(n-1)d+1}\subseteq A_{1}+B_{(n-1)d}.
\ee
In view of $|J_2|\geq 2$, we have $|B_{(n-1)d}|=|B_{(n-1)d+1}|$; moreover, $|A_0|\in \{|A_{1}|,\,|A_{1}|-1\}$ with $|A_0|=|A_{1}|-1$ only possibly when $d=1$. Thus, as all sets $A_i$ and $B_j$ are arithmetic progressions of common difference, we see that
$$|A_0+B_{(n-1)d+1}|\in \{|A_{1}+B_{(n-1)d}|,\,|A_{1}+B_{(n-1)d}|-1\}$$ with equality only possible when $d=1$. Consequently,
\eqref{stable}  implies $a_0+b_{(n-1)d+1}\in a_{1}+b_{(n-1)d}+\{0,1\}$ with $a_0+b_{(n-1)d+1}= a_{1}+b_{(n-1)d}+1$ only possible for $d=1$. In view of \eqref{eq:aligned} and the definition of $\alpha$, this means $\alpha=b_{(n-1)d+1}-b_{(n-1)d}\in\{0,1\}$ with $\alpha=1$ only possible for $d=1$. If $\alpha=0$, then $A$ and $B$ are standard trapezoids with slopes $0$ and $d$, yielding (i) of the claim. If $\alpha=1$ and $d=1$, then applying the affine transformation given by $\varphi(x,y)=(x-y,y)$ results in $A$ and $B$ being standard trapezoids with slopes $0$ and $-1$, also yielding  part
(i) of the claim and completing Case 2.

\subsection*{Case 3: } Either $|I_2|\ge 2$ and $|J_2|=1$ or else $|I_2|=1$ and $|J_2|\geq 2$. By symmetry, it suffices to consider the case  $|I_2|\ge 2$ and $|J_2|=1$. Then, since $B$ is two-dimensional and $c=0$, it follows that  $d\ge 1$.

Since $|I_2|\geq 2$, we see that
$I_2\setminus \{(m-1)d\}$  is nonempty. Let $i\in I_2\setminus \{(m-1)d\}$.
Then $i-1\in I_2$ as well, whence  $(i-1,(n-1)d)\in I_2\times J_2$ is tight. Hence \eqref{kl-subset} implies
\be\label{fust}
A_{i}+B_{(n-1)d-1}\subseteq A_{i-1}+B_{(n-1)d},
\ee
for each $i\in I_2\setminus \{(m-1)d\}$. Noting that $|A_{i}+B_{(n-1)d-1}|= |A_{i-1}+B_{(n-1)d}|-1$, it follows from \eqref{fust} that
 $a_{i}+b_{(n-1)d-1}=a_{i-1}+b_{(n-1)d}+\{0,1\}$. Combining with \eqref{eq:aligned} shows that
  $$\epsilon_{i}:=a_{i}-a_{i-1}\in
\{0,1\},$$ for $i\in [(m-1)d+1,h-1]$. Set $\epsilon_i=0$ for $i\in \Z$ outside the interval $[(m-1)d+1,h-1]$.

Suppose $\epsilon_{(m-1)d+j}\neq 0$ for some $j\in [1,d-1]$ and assume $j$ is minimal with this property. Then $d\geq 2$, else the interval $[1,d-1]$ is empty. In view of the minimality of $j$ and \eqref{eq:aligned}, it follows that  $\epsilon_{(m-1)d+j}=a_{(m-1)d+j}=1$.
However, since the pairs $((m-1)d,(n-1)d-1)\in I_1\times J_1$ and $((m-1)d+j,(n-1)d-1-j)\in [0,md-1]\times [0,(n-1)d-d]$ are both  tight by the modular rule, it follows that \eqref{kl-subset} implies $A_{(m-1)d}+B_{(n-1)d-1}=A_{(m-1)d+j}+B_{(n-1)d-1-j}.$ Thus, making use of \eqref{eq:aligned}, it follows that  $$0=a_{(m-1)d}+b_{(n-1)d-1}=a_{(m-1)d+j}+b_{(n-1)d-1-j}=a_{(m-1)d+j},$$ contradicting that $a_{(m-1)d+j}=1$ as just shown above. So, since $\epsilon_i=0$ for $i\leq (m-1)d$ by definition, we instead conclude that $\epsilon_{i}=0$ for all $i\leq md-1$.

To show that part (ii) of Claim 1 holds for $A$ and $B$, it remains to show that, for the sequence $\epsilon=(\epsilon_i:\,i\in \Z)$, every
subsequence of $d$ consecutive terms has at most one entry equal to $1$, as then $A$ will be  an $\epsilon$--standard trapezoid
$T_{\epsilon}(m,h,0,d)$  with $B$  a
standard trapezoid $T(n,(n-1)d+1,0,d)$. So suppose by contradiction that $\epsilon_i=1$ and $\epsilon_j=1$ with $i,\,j\in I_2$, \ $i<j$ and $j-i\leq d-1$. Note this is only possible if $d\geq 2$. We may assume the difference $j-i$ is minimal, so that $\epsilon_s=0$ for $s\in [i+1,j-1]$. Since $\epsilon_{(m-1)d}=0$ and $\epsilon_i=1$ with $i\in I_2$, we have $i\geq (m-1)d+1$, so that $i-1\in I_2$. By the definition and minimality of $i$ and $j$, we have $a_{j}=a_{i-1}+2$. Since $(i-1,(n-1)d)\in I_2\times J_2$, it follows that this pair is tight, whence \eqref{kl-subset} yields $A_{j}+B_{(n-1)d-(j-i+1)}\subseteq
A_{i-1}+B_{(n-1)d}$. As a result, since
$|A_{j}+B_{(n-1)d-(j-i+1)}|=|A_{(m-1)d}+B_{(n-1)d}|-1$ in view of $j-i+1\leq d$, we conclude from \eqref{eq:aligned} that $a_{j}\in a_{i-1}+\{0,1\}$, contradicting that $a_{j}=a_{i-1}+2$ and completing the proof of Claim 1.
\end{proof}

By applying the linear transformation $(x,y)\mapsto (x,-y)$ and translating, we
deduce the analogous result when $d=0$ by swapping the roles of $c$
and $d$ in Claim \ref{claim:c=0}. To complete the proof of the
theorem, we use the following claim.

\begin{claim}\label{claim:part} Let $(A,B)$ be an extremal pair for
\eqref{eq:eq-discreteII}. Let $t$ and $t'$ be such that $|A_t|=m$
and  $|B_{t'}|=n$ respectively,  and set $$A^-=\cup_{i\le t} A_i,\;
B^-=\cup_{j\le t'} B_j,\; A^+=\cup_{i\ge t} A_i,\; B^+=\cup_{j\ge t'}
B_j.$$ Then each of $(A^-, B^-)$ and $(A^+, B^+)$ are   extremal pairs
for \eqref{eq:eq-discreteII}.
\end{claim}

\begin{proof} Since $(A,B)$ is an extremal pair for  \eqref{eq:eq-discreteII}, we know that
  $A_i$ and $B_j$ are arithmetic progressions with the same common difference for each $i$ and $j$.
  We have $A^-\cap A^+=A_t$, $B^-\cap B^+=B_{t'}$ and $(A^-+B^-)\cap(A^++B^+)=A_t+B_{t'}$.
Using $|A_t+B_{t'}|=|A_t|+|B_{t'}|-1=m+n-1,$ we get
 \begin{eqnarray*}
  |A+B|&\ge & |(A^-+B^-)\cup (A^++B^+)|\\
  &=&|A^-+B^-|+|A^++B^+|-|A_t+B_{t'}|\\
  &\ge &\left(
  \left(\frac{|A^-|}{m}+\frac{|B^-|}{n}-1\right)+\left(\frac{|A^+|}{m}+\frac{|B^+|}{n}-1\right)-1\right)(m+n-1)\\
  &=&\left(\frac{|A|}{m}+\frac{|B|}{n}-1\right)(m+n-1),
  \end{eqnarray*}
and all the inequalities are equalities, proving the claim.
\end{proof}

If either $c=0$ or $d=0$, then the theorem follows from Claim
\ref{claim:c=0}. Therefore assume $c\leq -1$ and $d\geq 1$. \
Let $$t_d=\max I_2,\quad t_c=\min I_2=(m-1)d,\quad t'_d=\max J_2\;\und\; t'_c=\min J_2=(n-1)d$$ and, for $t\in I_2$ and $t'\in J_2$, define
$$A^-(t)=\cup_{i\leq t}A_i,\quad B^-(t')=\cup_{i\leq t'}B_i,\quad
A^+(t)=\cup_{i\geq t}A_i,\;\und\; B^+(t')=\cup_{i\geq
t'}B_i.$$

We may assume, by exchanging the role of $A$ and $B$ as need be, that $|I_2|\geq |I_1|$ and, by a possible vertical reflection, that $d\geq |c|$.
By Claim 2,  $(A^-(t_d),B^-(t'_d))$ is an extremal pair for
\eqref{eq:eq-discreteII} with $c=0$. Thus, applying an appropriate affine transformation from Claim 1 coupled with a possible horizontal reflection, we can assume (i) or (ii) from Claim 1 holds for $(A^-(t_d),B^-(t'_d))$ with the nonzero slope of the $\epsilon$-standard or standard trapezoid being positive and with \be\label{special-extent}a_i=b_j=0\quad \mbox{ for all }\; i\in [0,md-1]\cap (I_1\cup I_2)\und j\in [0,nd-1]\cap (J_1\cup J_2),\ee where we have made use of the fact that $\epsilon_i=0$ for the first $d$ entries of $I_2$ by definition of an $\epsilon$-standard trapezoid.
By Claim 2, $(A^+(t_c),B^+(t'_c))$ is also an extremal pair for \eqref{eq:eq-discreteII} with $d=0$. Thus we can apply Claim 1 to the vertical reflections of $A^+(t_c)$ and $B^+(t'_c)$ to conclude there exists
a linear transformation $\phi$  of
the form $(x,y)\mapsto (x-\beta y, -y)$ such that $\phi(A^+(t_c))$ and $\phi(B^+(t'_c))$
satisfy one of the conclusions (i) or (ii) of Claim 1.

Suppose that $|I_2|, |J_2|\ge 2$. Then part (i) of Claim 1 must hold for both
$(A^-(t_d),B^-(t'_d))$ and $(\phi(A^+(t_c)),\phi(B^+(t'_c))$, in which case $A^+(t_c)$ and $B^+(t_c)$ are (non-standard) trapezoids. Since $|I_2|, |J_2|\ge 2$, there is a unique linear transformation of the form $(x,y)\mapsto (x-\alpha y,-y)$ that maps the trapezoids $A^+(t_c)$ and $B^+(t'_c)$ to a pair of standard trapezoids, i.e., there is a unique such linear transformation that makes the parallel sides vertical. Hence, since the parallel sides of the standard trapezoid $A^-(t_d)$, which are already vertical by assumption, are also the parallel sides of the (non-standard) trapezoid $A^+(t_c)$, we conclude that $\phi(x,y)=(x,-y)$ is simply a vertical reflection, in which case both $A$ and $B$ are standard trapezoids with common slopes, yielding (a).
So, in view of the assumption $|I_2|\geq |J_2|$, we can instead assume  $|J_2|=1$, whence $t'_c=t'_d=(n-1)d$.

Regardless of whether (i) or (ii) holds for $\phi(A^+(t_c))$ and $\phi(B^+(t'_c))$, it follows that $\phi(A^+(t_d))$ and $\phi(B^+(t'_d))$ are standard trapezoids (in fact, triangles) $T(m,1,\pm c,0)$ and $T(n,1,\pm c,0)$.
The sign of the difference $\pm c$ depends, of course, on whether $(a_i:i\in I_3)$ and $(b_i:i\in J_3)$ or $(a'_i:i\in I_3)$ and $(b'_i:i\in J_3)$ are arithmetic progressions of common difference, one of which must hold, with both signs possible  if and only if $|c|=1$ which occurs if and only if both pairs of sequences consist of two arithmetic progressions of common difference. The difference is $|c|=-c$ when $(a'_i:i\in I_3)$ and $(b'_i:i\in J_3)$ are arithmetic progressions, and $-|c|=c$ when $(a_i:i\in I_3)$ and $(b_i:i\in J_3)$ are arithmetic progressions. Moreover, since the map $\phi$ must either take $(a_i:i\in I_3)$ or $(a'_i:i\in I_3)$ to an arithmetic progression with difference $0$, it follows that $\phi$ is uniquely defined by this property when $c\leq -2$.

The pair $(t_d,t'_d)=(t_d,(n-1)d)\in I_2\times J_2$ is tight, whence \eqref{kl-subset}  and \eqref{special-extent} imply $$A_{t_d+1}+B_{(n-1)d-1}\subseteq A_{t_d}+B_{(n-1)d}\subseteq [a_{t_d},a_{t_d}+m+n-2]\times \{t_d+(n-1)d\}.$$ Since $B_{(n-1)d-1}=[0,n-2]\times \{(n-1)d-1\}$ and $|A_{t_d+1}|=m-1$, the previous inclusion implies that $(a_{t_d+1},a'_{t_d+1})\in \{(a_{t_d},a_{t_d}+m-2),(a_{t_d}+1,a_{t_d}+m-1),(a_{t_d}+2,a_{t_d}+m)\}$. We continue with the following claim.

\subsection*{Claim 3: } $(a_{t_d+1},a'_{t_d+1})=(a_{t_d}+2,a_{t_d}+m)$ implies $d=1$.
Suppose by contradiction that $d\geq 2$ and first consider the case when
 $|I_2|\geq 2$, so that $a_{(m-1)d+1}\in I_2$. Then, since $d\geq 2$, we conclude from \eqref{special-extent} that $A_{(m-1)d+1}=[0,m-1]\times\{(m-1)d+1\}$ and $A_{(m-1)d}=[0,m-1]\times\{(m-1)d\}$.
 As a result, since $(a_{(m-1)d+1},b_{(n-1)d})\in I_2\times J_2$ is tight,  \eqref{kl-subset} yields $$A_{(m-1)d}+B_{(n-1)d+1}\subseteq A_{(m-1)d+1}+B_{(n-1)d}=[0,m+n-2]\times \{(m+n-2)d+1\}.$$
Since $A_{(m-1)d}=[0,m-1]\times\{(m-1)d\}$ and $|B_{(n-1)d+1}|=|B_{t'_d+1}|=n-1$, the above inclusion implies
$(b_{(n-1)d+1},b'_{(n-1)d+1})=(b_{t'_d+1},b'_{t'_d+1})\in \{(0,n-2),(1,n-1)\}$. However, making use of the assumption $(a_{t_d+1},a'_{t_d+1})=(a_{t_d}+2,a_{t_d}+m)$, we see that
if $(a_i:i\in I_3)$ and $(b_i:i\in J_3)$ are   arithmetic progression with common difference $a_{t_d+1}-a_{t_d}=2$, then $(b_{t'_d+1},b'_{t'_d})=(2,n)$, while if $(a'_i:i\in I_3)$ and $(b'_i:i\in J_3)$ are   arithmetic progression with common difference $a'_{t_d+1}-a'_{t_d}=a'_{t_d+1}-(a_{t_d}+m-1)=1$, then  $(b_{t'_d+1},b'_{t'_d})=(3,n)$, both contradicting that we just showed $(b_{t'_d+1},b'_{t'_d+1})\in \{(0,n-2),(1,n-1)\}$. So it remains to consider the case $|I_2|=1$, whence $t_d=t_c=(m-1)d$ and $a_{t_d}=0$ by \eqref{special-extent}.

In this case, $(t_d,(n-1)d-1)=(t_c,(n-1)d-1)\in I_2\times J_2$ is tight by the modular rule, whence \eqref{special-extent} and \eqref{kl-subset} together imply $$A_{t_d+1}+B_{(n-1)d-2}\subseteq A_{t_d}+B_{(n-1)d-1}=[0,m+n-3]\times \{t_d+(n-1)d-1\}.$$ Since $d\geq 2$, \eqref{special-extent} implies that $B_{(n-1)d-2}=[0,n-2]\times\{(n-1)d-2\}$, whence the above inclusion yields
$$(a_{t_d+1},a'_{t_d+1})\in \{(0,m-2),(1,m-1)\}=\{(a_{t_d},a_{t_d}+m-2),(a_{t_d}+1,a_{t_d}+m-1)\},$$ contrary to the assumption of the claim.
This concludes the proof of Claim 3.

\medskip

Based upon the work done before Claim 3, we divide the remainder of the proof into five short cases depending on the value of $(a_{t_d+1},a'_{t_d+1})$ and whether $(a_i:i\in I_3)$ and $(b_i:i\in J_3)$ or $(a'_i:i\in I_3)$ and $(b'_i:i\in J_3)$ are arithmetic progressions of common difference.

\subsection*{Case 1: } $(a_{t_d+1},a'_{t_d+1})=(a_{t_d},a_{t_d}+m-2)$ and $(a_i:i\in I_3)$ and $(b_i:i\in J_3)$ are arithmetic progressions of common difference $a_{t_d+1}-a_{t_d}=0$.
 Now, since $(a_i:i\in I_3)$ and $(b_i:i\in J_3)$ are arithmetic progressions of common difference $a_{t_d+1}-a_{t_d}=0$, we see that \eqref{eq:aligned} holds for appropriate translations of the  vertical reflections of $A^+(t_c)$ and $B^+(t_c)$. Since $|J_1|=1$, this means, as shown in the proof of Claim 1, that we may assume $\phi$ is simply the vertical reflection map $\phi(x,y)=(x,-y)$. Consequently, if $\epsilon_i=0$ for all $i$, which is the case if (i) holds for $A^-(t_d),B^-(t'_d))$, then (a) immediately follows for $A$ and $B$. Otherwise, since the definition of an $\epsilon$-standard trapezoid ensures that  $\epsilon_i\geq 0$ for all $i$ with equality for $i$ outside $I_2$, and since $a_{t_c}=0$ by \eqref{special-extent}, we thus  have $a_{t_d}>0=a_{t_c}$.

 Since $\phi$ is simply the vertical reflection map $\phi(x,y)=(x,-y)$, it follows that  $\phi(A^+(t_c))$ is an $\epsilon$-standard trapezoid which is now simply the vertical reflection of $A^+(t_c)$. However, as $\epsilon_i\geq 0$ is part of the definition of an $\epsilon$-standard trapezoid, it now follows that $a_t\geq a_{t_d}$ for all $t\in [t_c,t_d]$, which, in view of $a_{t_c}=0<a_{t_d}$, is a contradiction.

\subsection*{Case 2: } $(a_{t_d+1},a'_{t_d+1})=(a_{t_d},a_{t_d}+m-2)$ and $(a'_i:i\in I_3)$ and $(b'_i:i\in J_3)$ are arithmetic progressions of common difference $a'_{t_d+1}-a'_{t_d}=
a'_{t_d+1}-(a_{t_d}+m-1)=-1$. If $c=-1$, then $(a_i:i\in I_3)$ and $(b_i:i\in J_3)$ are also arithmetic progressions of common difference $a_{t_d+1}-a_{t_d}=0$, which was a case already handled in Case 1. Therefore assume $c\leq -2$, whence \eqref{special-extent} and the assumption of the case together imply $b_{t'_d+2}=b_{(n-1)d+2}=-1$.

Since $((m-1)d,(n-1)d)\in I_2\times J_2$ is tight, \eqref{kl-subset} and \eqref{special-extent} imply
$$A_{(m-1)d-2}+B_{(n-1)d+2}\subseteq A_{(m-1)d}+B_{(n-1)d}=[0,m+n-2]\times \{(m+n-2)d\}.$$
In view of $d\geq |c|\geq 2$ and \eqref{special-extent}, we have $A_{(m-1)d-2}=[0,m-2]\times\{(m-1)d-2\}$, which combined with the above inclusion implies $b_{(n-1)d+2}\geq 0$, contrary to what we previously showed.

\subsection*{Case 3: } $(a_{t_d+1},a'_{t_d+1})=(a_{t_d}+2,a_{t_d}+m)$. Then, in view of Claim 3 and $d\leq |c|$, it follows that $d=1$ and $c=-1$, so that both $(a_i:i\in I_3)$ and $(b_i:i\in J_3)$ as well as $(a'_i:i\in I_3)$ and $(b'_i:i\in J_3)$ are arithmetic progressions of common differences $a_{t_d+1}-a_{t_d}=2$ and $a'_{t_d+1}-a'_{t_d}=
a'_{t_d+1}-(a_{t_d}+m-1)=1$, respectively, in which case applying the linear transformation $(x,y)\mapsto (-(x-y),y)$ reduces  Case 3 to the already considered Case 1.

\subsection*{Case 4: } $(a_{t_d+1},a'_{t_d+1})=(a_{t_d}+1,a_{t_d}+m-1)$ and $(a'_i:i\in I_3)$ and $(b'_i:i\in J_3)$ are arithmetic progressions of common difference $a'_{t_d+1}-a'_{t_d}=
a'_{t_d+1}-(a_{t_d}+m-1)=0$. In this case, similar to the argument in Case 1, we can assume, as shown in Claim 1, that  $\phi(x,y)=(-x,y)$ is simply a vertical reflection, and then (a) or (b) immediately follows for $A$ and $B$.

\subsection*{Case 5: } $(a_{t_d+1},a'_{t_d+1})=(a_{t_d}+1,a_{t_d}+m-1)$ and $(a_i:i\in I_3)$ and $(b_i:i\in J_3)$ are arithmetic progressions of common difference $a_{t_d+1}-a_{t_d}=1$. If $c=-1$, then $(a'_i:i\in I_3)$ and
$(b'_i:i\in J_3)$ are also arithmetic progressions of common difference $a'_{t_d+1}-a'_{t_d}=
a'_{t_d+1}-(a_{t_d}+m-1)=0$, which was a case already handled in Case 4. Therefore assume $c\leq -2$, in which case the map $\phi$ taking $A^+(t_c)$ to an $\epsilon$-standard or standard trapezoid is uniquely defined and must be given by $\phi(x,y)=(x-y,y)$. The conditions on the $\epsilon$-standard or standard trapezoid $A^+(t_c)$, when translated {\it with extreme care} back to $A$ via $\phi^{-1}$, imply that $\epsilon_i\in \{1,0\}$ for $i\in [(m-1)d+1,h-1]$, with no subsequence of $|c|$ consecutive terms $\epsilon_i$ with $i\in [(m-1)d+1,t_d]$ containing more than one value equal to $0$, and $\epsilon_i=1$ for $i\in [\max\{(m-1)d+1,t_d-|c|+2\},h-1]$. Since $A^-(t_d)$ is an $\epsilon$-standard or standard trapezoid with slopes $0$ and $d$, we know that no subsequence of $d$ consecutive terms $\epsilon_i$ with $i\in I_1\cup I_2$ contains more than one value equal to $1$, and that $\epsilon_i=0$ for $i\in [1,\min\{t_d,md-1\}]$.

Suppose $|I_2|\geq 2$. Then, since $|c|\geq 2$, it follows that $$t_d\in [\max\{(m-1)d+1,t_d-|c|+2\},h-1],$$ whence $\epsilon_{t_d}=1$ as noted in the previous paragraph. Thus, since $\epsilon_i=0$ for all $i\in [1,\min\{t_d,md-1\}]$, we conclude that $t_d\geq md$ and that $\epsilon_i=0$ for $i\in [1,md-1]$. If $d\geq 3$, then $\epsilon_i=0$ for $i=(m-1)d+1$ and $i=(m-1)d+2$, which contradicts that no consecutive subsequence of $|c|\geq 2$ terms $\epsilon_i$ with $i\in [(m-1)d+1,t_d]$ contains more than one value equal to $0$. Therefore $d=2$, whence $2\geq d\geq |c|\geq 2$ implies $c=-2$. Furthermore, since within the interval $[(m-1)d+1,t_d]$ there can be no $d=2$ consecutive terms $\epsilon_i$ equal to zero nor $|c|=2$ consecutive terms $\epsilon_i$ equal to $1$, we conclude that $\epsilon_i$ alternates between the values of $0$ and $1$ in the interval $[(m-1)d+1,t_d]$. Since the last term $\epsilon_{t_d}=1$, as noted before,  and since the first term $\epsilon_{(m-1)d+1}=0$, in view of $(m-1)d+1\leq md-1$, it follows that there are an even number of terms in $[(m-1)d+1,t_d]$, whence $|I_2|=|[(m-1)d+1,t_d]|+1$ is odd and (c) is now seen to hold. So it remains to consider the case when $|I_2|=1$.

In this case, if $d=2$, then $2=d\geq |c|\geq 2$ implies $c=-2$, and (c) is again seen to hold with $k=1$. Therefore assume $d\geq 3$. Since $((m-1)d-1,(n-1)d)\in I_2\times J_2$ is tight by the modular rule, \eqref{kl-subset} and \eqref{special-extent} imply
\be\label{einmal}A_{(m-1)d-3}+B_{(n-1)d+2}\subseteq A_{(m-1)d-1}+B_{(n-1)d}=[0,m+n-3]\times\{(m+n-2)d-1\}.\ee However, since $d\geq 3$ and $|c|\geq 2$, and since $(b_i:i\in J_3)$ is an arithmetic progression with difference $1$ by assumption of the case, it follows in view of \eqref{special-extent} that $A_{(m-1)d-3}+B_{(n-1)d+2}=[2,m+n-2]\times \{(m+n-2)d-1\}$, contradicting \eqref{einmal} and completing the proof.
\end{proof}

\section{The Continuous Case}\label{sec:continuous}

Let $A$ and $B$ be two convex bodies in
$\R^2$, meaning that $A$ and $B$ are compact convex sets with
nonempty interiors.  As we already mentioned in
the introduction,  Bonnesen proved the following inequality:
\begin{equation}\label{eq:bm2ibis}
|A+B| \ge \left( \frac{|A|}{m}+\frac{|B|}{n}\right) (m+n),
\end{equation}
where $|X|$ stands for the area of a measurable set $X$ in $\R^2$
and $m=|\pi (A)|$ and $n=|\pi(B)|$ are the lengths of the
projections of $A$ and $B$ onto the first coordinate. As was also pointed
out in  the introduction, Bonnesen's
inequality \eqref{eq:bm2ibis}  implies the classical Brunn-Minkowski inequality
\eqref{eq:bmi}.

In this section, we will describe the structure of convex  sets $A$
and $B$ in the plane for which the {\it Bonnesen equality}
\begin{equation}\label{eq:bm2}
|A+B| = \left( \frac{|A|}{m}+\frac{|B|}{n}\right) (m+n)
\end{equation}
holds.

A convex body $A\subseteq \R^2$ can be described as
\begin{equation}\label{eq:convex}
A=\{(x,y)\in \R^2: x\in \pi(A), u_A(x)\le y \le v_A(x)\},
\end{equation}
for some pair $u_A$ and $v_A$ of real functions defined on the
interval $\pi(A)$, where $\pi:\R^2\rightarrow \R$ denotes the orthogonal projection onto the $x$-axis,  $v_A$ is concave and $u_A$ is convex. Without loss of
generality, we may assume that $\pi (A)=[0,m]$ with $m>0$. {We will use
this notation in what follows.}

In order to formulate our main result, we need the following
definition. We say that $A'$ is a vertical {\it stretching} of $A$
of amount $h \ge 0$  if
$$
A'=\{(x,y)\in \R^2: x\in \pi(A)=[0,m], u_A(x)\le y \le v_A(x)+h\}.
$$
Let us show that, by stretching an extremal pair for the Bonnesen
inequality, we get another extremal pair.

\begin{lemma}\label{lem:stretching}
Let $A$ and $B$ be two convex bodies and let $A'$ be a vertical stretching
of $A$. Then $ (A, B)$ is an extremal pair for Bonnesen
inequality, i.e.,
\begin{equation*}
|A+B|=(\frac{|A|}{m}+\frac{|B|}{n})(m+n),
\end{equation*}
if and only if $ (A', B)$ is also an extremal pair, i.e.,
\begin{equation}\label{eq:D}
|A'+B|=(\frac{|A'|}{m}+\frac{|B|}{n})(m+n).
\end{equation}
\end{lemma}

\begin{proof} For a convex body $X$, we denote by $u_X$ and $v_X$ the
bottom and top functions which define $X$ as in \eqref{eq:convex}.
We observe that, if $X'$ is an vertical stretching of  $X$ of amount $h$, then
$$|X'|=\int_0^x (v_X(t)+h-u_X(t) )dt
=|X|+hx.$$

We assume that $A$ is given by \eqref{eq:convex} and $B=\{(x,y)\in
\R^2: 0 \le x \le n, u_B(x)\le y \le v_B(x)\}.$ If $u_{A+B}$ and $v_{A+B}$ are the bottom and top functions defining the convex body
$A+B$, then
$$
A'+B=\{ (x,y)\in \R^2:
0\le x\le m+n,
~~u_{A+B}(x)\le y\le v_{A+B}(x)+h\}.$$
This means that if $A'$ is a vertical stretching of $A$ of amount $h$, then $A'+B$ is
also a vertical stretching of $A+B$ of amount $h$. Therefore $|A'|=|A|+hm,
|A'+B|=|A+B|+h(m+n)$, and we get
\begin{eqnarray*}
|A'+B|-\left(\frac{|A'|}{m}+\frac{|B|}{n}\right)(m+n)
&=&|A+B|+h(m+n)-\left(\frac{|A|+hm}{m}+\frac{|B|}{n}\right)(m+n)\nonumber\\
&=&|A+B|-\left(\frac{|A|}{m}+\frac{|B|}{n}\right)(m+n).
\end{eqnarray*}
Lemma \ref{lem:stretching} follows.
\end{proof}

The characterization of extremal sets for the Bonnesen inequality
is the following theorem.

\begin{theorem}\label{thm:invcont-general} Let
\begin{eqnarray}
A&=&\{ (x,y)\in \R^2: 0\le x\le m, ~~u_A(x)\le y\le v_A(x)\},\label{eq:A}\\
B&=&\{ (x,y)\in \R^2: 0\le x\le n, ~~u_B(x)\le y\le
v_B(x)\}\label{eq:B}
\end{eqnarray}
be two convex bodies in the plane $\R^2.$ Then
$$|A+B|= \left( \frac{|A|}{m}+\frac{|B|}{n}\right)(m+n)$$
if and only if there is a pair $A'$ and $B'$ of
homothetic convex bodies such that $A$ is a vertical stretching of $A'$
and $B$ is a vertical stretching of $B'$.
\end{theorem}

The proof of Theorem \ref{thm:invcont-general} will be derived
from the results below. Its proof requires only basic notions
from the differential calculus of convex functions; see
\cite{convexcalculus}, \cite{convexcalculusB}. We summarize the needed points below. Let
$f$ be  a positive real concave function defined on an interval
$[m,n]$. We let
$$f_+'(x):=\underset{\lambda>0}{\lim_{\lambda\rightarrow 0}}\frac{f(
x+\lambda   )-f( x)}{\lambda} \und
~f_-'(x):=\underset{\lambda<0}{\lim_{\lambda\rightarrow 0}}\frac{f(
x+\lambda   )-f( x)}{\lambda}$$ denote the {\it right derivative}
and  {\it left derivative} of  $f$ at $x$ respectively. It is a
basic property of concave functions that these one sided derivatives
exist at every point $x\in (m,n)$ and that $f_-'(x)\ge f'_+(x)$, with
equality occurring precisely when $f$ is differentiable at $ x$.
When $f$ is concave, it is differentiable a.e. with $f'$ continuous
on the subset of points where it is defined. In fact, $f$ is
Lipschitz continuous, and thus absolutely continuous, so that the
Fundamental Theorem of Calculus holds. In particular, if the
derivative is zero a.e., then $f$ must be a constant
function.

We will first give the characterization of equality in Bonnesen's
inequality  for convex sets defined by the graph
of   non-negative concave functions as in \eqref{p-c1} and
\eqref{p-c2}.

\begin{theorem}\label{thm:invcont-graphs} Let $f$ and $g$ be concave  real positive
functions defined in the
intervals $[0,m]$ and $[0,n]$ respectively and let
\begin{eqnarray}
A&=& \{ (x,y)\in \R^2: 0\le x\le m, ~0\le y\le f(x) \} \label{p-c1},\\
B&=& \{ (x,y)\in \R^2: 0\le x\le n, ~0\le y\le g(x) \} \label{p-c2}
\end{eqnarray}
be convex sets in the plane $\R^2.$
Assume that $C=m\left(\frac{|A|}{m^2}-\frac{|B|}{n^2}\right)$ is non-negative.
Then, the pair $(A,B)$ is extremal for
Bonnesen inequality, that is,
$$
|A+B|=\left(\frac{|A|}{m}+\frac{|B|}{n}\right)(m+n),
$$
if and only if $$f(x)=\frac{m}{n}g(\frac{n}{m}x)+C,  \quad \mbox{ for every }\; 0 \le x \le m. $$
\end{theorem}

\noindent{\bf Remark.} If the constant
$C=m\left(\frac{|A|}{m^2}-\frac{|B|}{n^2}\right)$ is negative,
then we use a similar argument as in the proof of Theorem
\ref{thm:invcont-graphs} and we get that that the pair $(A,B)$ is
extremal for Bonnesen inequality if and only if the set $B$ is be
a vertical stretching of $A$.

For the proof of Theorem \ref{thm:invcont-graphs}, we proceed with a series of lemmas.
 The first shows that the condition on $f$ in Theorem
\ref{thm:invcont-graphs} is sufficient.

\begin{lemma}\label{lem:equality}
Let $A$ and $B$ be the convex sets in the plane $\R^2$ defined in
\eqref{p-c1} and \eqref{p-c2}, respectively.

If $C=m\left(\frac{|A|}{m^2}-\frac{|B|}{n^2}\right)\ge 0$
and $f(x)=\frac{m}{n}g(\frac{n}{m}x)+C,$
then
$$
|A+B|=\left(\frac{|A|}{m}+\frac{|B|}{n}\right)(m+n).
$$
\end{lemma}

\begin{proof} The hypothesis
$f(x)=\frac{m}{n}g(\frac{n}{m}x)+C$ implies that $A$ is a vertical
stretching of a homothetic copy of $B$ of amount $C$. Thus the
statement of the lemma follows from the case of equality in
Brunn--Minkowski Theorem and Lemma \ref{lem:stretching}, because
the lower bounds for $|A+B|$ implied by \eqref{eq:bmi} and
\eqref{eq:bm2i} coincide when $C=0$.
 \end{proof}


\begin{lemma}\label{lem:necessity1}
Let $A$ and $B$ be the convex sets in the plane $\R^2$ defined in
\eqref{p-c1} and \eqref{p-c2}, respectively. Then
\item {(a)}
$$|A+B|\ge  \left( \frac{|A|}{m}+\frac{|B|}{n}\right) (m+n)+\Delta,$$
where $$\Delta= \left(nf(m)-\frac{n}{m}\int_0^m
f(x)\;dx\right)+\left(m g(0)-\frac{m}{n}\int_0^n g(x)
\;dx\right).$$
\item {(b)} In particular, if $f'_+(x)\geq
g'_+(y)+\epsilon$ for all $x\in [0,m)$ and $y\in [0,n)$, where
$\epsilon\geq 0$, then
$$|A+B|\geq
\left(\frac{|A|}{m}+\frac{|B|}{n}\right)(m+n)
+\frac{mn}{2}\epsilon.$$
\end{lemma}

\begin{proof} We first observe that
\ber \nn A+(\{0\} \times [0,g(0)] )&\subseteq& (A+B)\cap
([0,m]\times\R),\\
\nn (\{m\} \times [0,f(m)] )+B&\subseteq& (A+B)\cap ([m,m+n]\times
\R ). \eer

Therefore, 
 \ber\nn |A+B|&\geq&|A+(\{0\} \times [0,g(0)])
|+|(\{m\} \times [0,f(m)])+B|
\\
&=& |A|+m g(0)+|B|+nf(m), \label{old-estimate} \eer
 and we complete
the proof of the first part of Lemma \ref{lem:necessity1} as
follows:
\begin{align}
&|A+B|-(\frac{|A|}{m}+\frac{|B|}{n} )(m+n)\geq \\\nn
&\geq |A|+mg(0)+|B|+nf(m)-(\frac{|A|}{m}+\frac{|B|}{n})(m+n)\\
&=mg(0)+nf(m)-\frac{n}{m}|A|-\frac{m}{n}|B|\\
&=
n\left(f(m)-\frac{|A|}{m}\right)+m\left(g(0)-\frac{|B|}{n}\right)=\Delta.
\end{align}

\medskip

It remains to prove assertion (b). Thus suppose $f'_+(x)\geq
g'_+(y)+\epsilon$ for all $x\in [0,m)$ and $y\in [0,n)$, where
$\epsilon\geq 0$. Since $f:[0,m]\rightarrow \R_{\geq 0}$ is a
concave function, and thus absolutely continuous and differentiable
a.e., it follows that $x f(x):[0,m]\rightarrow \R_{\geq 0}$ is also
absolutely continuous and differentiable a.e. Thus it follows from the
Fundamental Theorem of Calculus that
$$f(m)=\frac{1}{m}\int_{0}^m (x f(x))'\;dx=\frac{1}{m}\int_{0}^m x
f'(x)\;dx+\frac{1}{m}\int_{0}^m f(x)\;dx=\frac{1}{m}\int_{0}^m x
f'(x)\;dx+\frac{|A|}{m}.$$  Hence we may rewrite $\Delta$ as
\begin{eqnarray}
\Delta&=&\frac{n}{m}\int_0^m y f'(y)\;dy-\frac{m}{n}\int_0^n
g(x)\;dx
+m g(0)\nn\\
&=&\frac{m}{n}\int_0^n \left(x
f'_+(\frac{m}{n}x)-g(x)\right)\;dx+m g(0)\label{stickstep},
\end{eqnarray} where we have used that $f'_+(x)=f'(x)$ a.e. and
applied the change of variables $y\mapsto \frac{m}{n}x$.

Since  $f'_+(x)\geq g'_+(y)+\epsilon$ for all $x\in [0,m)$ and
$y\in [0,n)$, it follows that
\be\label{fun0}xf'_+(\frac{m}{n}x)\geq x g'_+(0)+x\epsilon,\ee for
all $x\in [0,n)$. Since $g$ is concave, $\frac{g(x)-g(0)}{x}$ is a
decreasing function of $x$, whence
\be\label{fun1}g'_+(0)=\underset{\lambda>0}{\lim_{\lambda\to
0}}\frac{g(\lambda)-g(0)}{\lambda}\geq \frac{g(x)-g(0)}{x},\ee for
all $x\in [0,n)$. Applying the estimates \eqref{fun0} and
\eqref{fun1} to \eqref{stickstep}, we obtain \ber\nn\Delta&\geq& m
g(0)+\frac{m}{n}\int_0^n (x
g'_+(0)+x\epsilon-g(x))\;dx\\
&\geq& m g(0)+\frac{m}{n}\int_0^n
(-g(0)+x\epsilon)\;dx=\frac{mn}{2}\epsilon, \eer
 completing the proof.
\end{proof}

In order to complete the proof of Theorem \ref{thm:invcont-graphs},
we need the following lemma.

\begin{lemma}\label{lem:partition}
Let $A$ and $B$ be the convex sets in the plane $\R^2$ defined by
\eqref{p-c1} and \eqref{p-c2}, respectively.

Let $\Pi_{m,k}=\{x_0=0<x_1<\cdots <x_k=m\}$ and
$\Pi_{n,k}=\{x_0'=0<x_1'<\cdots <x_k'=n\}$ be partitions of $[0,m]$ and $[0,n]$ into $k$
equal parts,  and let
\begin{eqnarray}
A_i&=& \{ (x,y)\in \R^2: x_{i-1}\le x\le x_i, ~0\le y\le f(x) \}, ~i=1,\ldots ,k, \nn\\
B_i&=& \{ (x,y)\in \R^2: x'_{i-1}\le x\le x'_i, ~0\le y\le g(x)
\}, ~i=1,\ldots ,k. \nn
\end{eqnarray}
If $(A,B)$ is an extremal pair for Bonnesen's inequality,
i.e. $|A+B|=\left(\frac{|A|}{m}+\frac{|B|}{n}\right)(m+n),$ then
each of $(A_1,B_1),...,(A_k,B_k)$ is also an extremal pair.
\end{lemma}

\begin{proof} Observe that $\cup_{i=1}^k (A_i+B_i)\subseteq A+B$ and that
$|(A_i+B_i)\cap (A_j+B_j)|=0$ for
$i\neq j$. It follows that
\begin{eqnarray*}
    |A+B|&\ge&\sum_{i=1}^k|A_i+B_i|\nn\\
    &\ge &
    \sum_{i=1}^k\left(\frac{|A_i|}{x_i-x_{i-1}}+\frac{|B_i|}{x'_i-x'_{i-1}}\right)
    \left((x_i-x_{i-1})+(x'_i-x'_{i-1})\right)\\
    &=&\sum_{i=1}^k
    \left(\frac{|A_i|}{m/k}+\frac{|B_i|}{n/k}\right)(m/k+n/k)=\left(\frac{|A|}{m}+\frac{|B|}{n}\right)(m+n),
\end{eqnarray*}
where we have applied Bonnesen inequality \eqref{eq:bm2ibis} for
the second inequality. Equality
$|A+B|=\left(\frac{|A|}{m}+\frac{|B|}{n}\right)(m+n)$ implies that
$|A_i+B_i|=\left(\frac{|A_i|}{m/k}+\frac{|B_i|}{n/k}\right)(m/k+n/k),$
for every $1 \le i \le k$. Lemma \ref{lem:partition} follows.
\end{proof}

\begin{proof}[Proof of Theorem \ref{thm:invcont-graphs}.] The sufficiency
of the condition on $f$ is Lemma \ref{lem:equality}. Let us show
the necessity.

Set $\lambda=m/n$. If $f'(x)=g'(\lambda^{-1}x)$ for a.e. $x\in (0,m)$, then, by the Fundamental
Theorem of Calculus, we would have $f(x)=\lambda g(\lambda^{-1}x)+c$ for
some constant $c$, whose value is easily computed to be $C$, and the result
follows.

Thus we may assume that there is an interior point $x_0\in (m,n)$
where both functions are differentiable and, by exchanging the roles
of $f$ and $g$ if necessary,   $f'(x_0)\ge g'(x_0)+ \epsilon'$ for
some $\epsilon'>0$. By continuity of the derivatives, there is an
$\epsilon>0$ and a $\delta>0$  such that $f'(x)\ge
g'(\lambda^{-1}x)+\epsilon$ a.e. in the interval $
[x_0-\delta,x_0+\delta]$. Choose $k$ sufficiently large so that, in the partition
$\{x_0=0<x_1<\cdots <x_k=m\}$ of  $[0,m]$ into $k$ equal parts, one
of the parts, say $[x_{i-1},x_i]$, is contained in the interval
$[x_0-\delta,x_0+\delta]$. Then $\{\lambda^{-1} x_0=0<\lambda^{-1} x_1<\cdots
<\lambda^{-1} x_k=n\}$ is  a partition of $[0,n]$ whose $i$-th part
$[\lambda^{-1} x_{i-1},\lambda^{-1} x_i]$  is contained in the interval
$[\lambda^{-1} (x_0- \delta), \lambda^{-1}(x_0+\delta)]$.

 By assertion (b) of
Lemma \ref{lem:necessity1}, the restrictions $A_i=A\cap
([x_{i-1},x_i]\times \R)$ and $B_i=B\cap ([\lambda x_{i-1},\lambda
x_i]\times \R)$ do not satisfy equality in Bonnesen's inequality. Hence, by
Lemma \ref{lem:partition}, the original sets $A$ and $B$ are not an
extremal pair. This completes the proof of Theorem
\ref{thm:invcont-graphs}.\end{proof}

We can now complete the proof of Theorem \ref{thm:invcont-general}.

\begin{proof}[Proof of Theorem \ref{thm:invcont-general}] Assume that $A$
and $B$ are two  convex bodies defined by (\ref{eq:A}) and
(\ref{eq:B}), respectively. Without loss of generality, we may assume
that
\begin{eqnarray}
&&u_A(x) \le 0 \mbox{ for } 0 \le x \le m \label{eq:negative-A} \und\\
&&u_B(x) \le 0 \mbox{ for } 0 \le x \le n \label{eq:negative-B}.
\end{eqnarray}
In view of Lemma \ref{lem:stretching} and the fact that homothetic convex bodies attain equality in the Brunn-Minkowski bound, and thus also in Bonnesen's bound, we see that the sets described by Theorem \ref{thm:invcont-general} attain the equality in Bonnesen's bound. It remains to show these are the only possibilities.

Choose $h >0$ large enough such that
\begin{eqnarray}
&&w_A(x)=v_A(x)+h \ge 0, \mbox{ for } 0 \le x \le m
\label{eq:positive-A},
\end{eqnarray}
and
\begin{eqnarray}
&&w_B(x)=v_B(x)+h \ge 0, \mbox{ for } 0 \le x \le n.
\label{eq:positive-B}
\end{eqnarray}
Let
\begin{eqnarray}
&& A_h=\{ (x,y)\in \R^2: 0\le x\le m, ~~u_A(x)\le y\le w_A(x)\},\nonumber\\
&& B_h=\{ (x,y)\in \R^2: 0\le x\le n, ~~u_B(x)\le y\le
w_B(x)\}.\nonumber
\end{eqnarray}
It follows from  (\ref{eq:negative-A}), (\ref{eq:negative-B}),
(\ref{eq:positive-A}), (\ref{eq:positive-B})   that we may split
each of the convex sets $A_h$ and $B_h$ as follows:
\begin{eqnarray}
&& A_h=A_h^+ \cup A_h^-,
\nonumber\\
&& A_h^-=\{ (x,y)\in \R^2: 0\le x\le m, ~~u_A(x)\le y\le 0\},\label{eq:Ahminus}\\
&& A_h^+=\{ (x,y)\in \R^2: 0\le x\le m, ~~0     \le y\le
w_A(x)\},\label{eq:Ahplus}
\end{eqnarray}
and
\begin{eqnarray}
&& B_h=B_h^+ \cup B_h^-,
\nonumber\\
&& B_h^-=\{ (x,y)\in \R^2: 0\le x\le n, ~~u_B(x)\le y\le 0\},\label{eq:Bhminus}\\
&& B_h^+=\{ (x,y)\in \R^2: 0\le x\le n, ~~0     \le y\le
w_B(x)\}.\label{eq:Bhplus}
\end{eqnarray}
We observe  that
\begin{eqnarray}
(A_h^+ + B_h^+)\cup (A_h^- + B_h^-)\subseteq A_h+B_h. \label{eq:++--}
\end{eqnarray}

By Lemma \ref{lem:stretching},  equality
\begin{equation}\label{eq:E}
|A+B|=(\frac{|A|}{m}+\frac{|B|}{n})(m+n)
\end{equation}
and equality
\begin{equation}\label{eq:F}
|A_h+B_h|=(\frac{|A_h|}{m}+\frac{|B_h|}{n})(m+n)
\end{equation}
are equivalent.

In view of  (\ref{eq:++--}), we have
\begin{eqnarray}
|A_h+B_h|\geq |A_h^++B_h^+|+|A_h^-+B_h^-|,\nonumber
\end{eqnarray}
and Bonnesen's inequality (\ref{eq:bm2}) gives
\begin{eqnarray}
|A_h+B_h|&\geq &|A_h^++B_h^+|+|A_h^-+B_h^-|\nonumber\\
&\ge&(\frac{|A_h^+|}{m}+\frac{|B_h^+|}{n})(m+n)+(\frac{|A_h^-|}{m}+\frac{|B_h^-|}{n})(m+n)\nonumber\\
&=&(\frac{|A_h|}{m}+\frac{|B_h|}{n})(m+n).\nonumber
\end{eqnarray}
Therefore (\ref{eq:F}) implies
\be
|A_h^++B_h^+|=(\frac{|A_h^+|}{m}+\frac{|B_h^+|}{n})(m+n) \mbox{ and
}
|A_h^-+B_h^-|=(\frac{|A_h^-|}{m}+\frac{|B_h^-|}{n})(m+n).\label{eq:new}
\ee
Consequently, $ (A_h^+, B_h^+)$ and $ (A_h^-, B_h^-)$ are extremal
pairs for Bonnesen's inequality. Moreover, the convex sets $A_h^+$ and
$B_h^+$ are included in the half plane given by $y \ge 0$, \ $A_h^-$ and
$B_h^-$ are included in the half plane given by $y \le 0$, and these sets are defined
by (\ref{eq:Ahminus}), (\ref{eq:Ahplus}), (\ref{eq:Bhminus}) and
(\ref{eq:Bhplus}). Thus we may apply Theorem \ref{thm:invcont-graphs} and
 conclude that (\ref{eq:new}) is true if and only if there are
$c', d\in \R$ such that, for every $0 \le x \le n$, we have
$$w_B(x)=\frac{n}{m}w_A(\frac{m}{n} x)+c' \und v_B(x)=\frac{n}{m}v_A(\frac{m}{n}
x)+d.$$ In view of the definition of $w_A$ and $w_B$, this shows that the curves $u_A$ and $u_B$,
as well as the curves $v_A$ and $v_B$, are homothetic.

Let $$\alpha=\inf_{x\in [0,m]}\{v_A(x)-u_A(x)\}\geq 0\und\beta=\inf_{x\in [0,n]}\{v_B(x)-u_B(x)\}\geq 0.$$
Let $x_0\in [0,m]$ and $y_0\in [0,n]$ be points
achieving these minima.
Let $A'\subseteq \R^2$ be the subset with $\pi(A')=\pi(A)=[0,m]$ defined by
 $$A'=\{(x, y) \in \R^2\mid  x\in [0,m],\; u_A(x)\le y \le v_A(x)-\alpha\}.$$
Then $u_{A'}=u_A$ and $v_{A'}=v_A-\alpha$ (in view of the
definition of $\alpha$) so that $A'$ is the maximal vertical
`compression' of $A$. In particular, $A$ is a vertical stretching
of $A'$ of amount $\alpha$. Likewise, let $B'\subseteq \R^2$ be
the subset with $\pi(B')=\pi(B)=[0,n]$ defined by  $$B'=\{(x, y)
\in \R^2\mid  x\in [0,n],\; u_B(x)\le y \le v_B(x)-\beta\}.$$
Since the  curves $v_A$ and $v_B$ are both homothetic as well as
the curves $u_A$ and $u_B$, it follows that we can take $x_0=
\frac{m}{n} y_0$ and, moreover, $A'$ and $B'$ will then be
homothetic convex bodies (note $v_A(x)$ and $u_A(x)$ intersect
over the point $x=x_0$ as do $v_B(\frac{m}{n}x)$ and
$u_B(\frac{m}{n}x)$, so that these curves fully determine the
convex body $A'$ as well as the dilation $\frac{m}{n}\cdot B'$),
which completes the proof.
\end{proof}

\end{document}